\documentclass[a4paper,12pt,oneside,reqno]{amsart}

\usepackage{geometry}
\usepackage{amsmath, amssymb, amsfonts, mathtools, tikz-cd}
\usepackage{enumitem}
\usepackage{fancyhdr}
\usepackage{stmaryrd}
\usepackage{circuitikz}
\usepackage{graphicx}
\usepackage{amsmath}
\usepackage{mathrsfs}
\usepackage{mathtools}
\setlist[enumerate]{label=\upshape(\roman*)}

\newcommand{\Proj}{\operatorname{Proj}} % Proj of a graded ring
\newcommand{\Pic}{\operatorname{Pic}} % Picard group
\newcommand{\Hom}{\operatorname{Hom}} % Hom set
\newcommand{\Ext}{\operatorname{Ext}} % Ext functor
\newcommand{\End}{\operatorname{End}} % Endomorphism ring
\newcommand{\NS}{\operatorname{NS}} % Neron--Severi group
\newcommand{\Cl}{\operatorname{Cl}} % Weil divisor class group
\newcommand{\Amp}{\operatorname{Amp}} % Ample cone
\newcommand{\rk}{\operatorname{rk}} % rk
\newcommand{\iso}{\xrightarrow{\smash{\raisebox{-0.5ex}{\(\textstyle\sim\)}}}}
\newcommand{\bwed}{\mathop{\raisebox{0.25ex}{\scalebox{0.95}{\(\bigwedge\)}}}\nolimits}
\DeclareMathOperator{\Db}{D\textsuperscript{\rm b}}

\newcommand{\vv}{\mathbf{v}} % Mukai vector

\newcommand{\gitq}{\mathord{/\!/}}
\newcommand{\sympq}{\mathord{/\!\!/\!\!/}}
\numberwithin{equation}{section}
\numberwithin{figure}{section}

\theoremstyle{plain}
\newtheorem{theorem}{Theorem}[section]
\newtheorem{lemma}[theorem]{Lemma}
\newtheorem{proposition}[theorem]{Proposition}
\newtheorem{corollary}[theorem]{Corollary}

\newtheorem{thmx}{Theorem}

\theoremstyle{definition}
\newtheorem{definition}[theorem]{Definition}
\newtheorem{example}[theorem]{Example}

\newtheorem{exercise*}[theorem]{Exercise*}

\theoremstyle{remark}
\newtheorem{remark}[theorem]{Remark}

\usepackage[pagebackref]{hyperref}

\definecolor{imperial}{rgb}{0.015625, 0.25, 0.4375}
\definecolor{denim}{rgb}{0.08, 0.38, 0.74}

\hypersetup{
    colorlinks=true,
    linkcolor=imperial,
    citecolor=imperial,
    filecolor=magenta,      
    urlcolor=denim
    }
    
\title[Irreducible symplectic varieties via K3-del Pezzo double covers]{Irreducible symplectic varieties via K3-del Pezzo double covers}
\author[Riccardo Carini]{Riccardo Carini}

\begin{document}

\begin{abstract}
  We construct a series of irreducible symplectic varieties of dimension \(2n\), for \(2\leq n\leq 10\), with second Betti numbers \(16\leq b_2\leq 24\). They arise as non-trivial terminalisations of finite symplectic quotients of Beauville--Mukai systems on very general K3-del Pezzo double covers.
\end{abstract}

\maketitle
\addtocontents{toc}{\protect\setcounter{tocdepth}{-1}}

\thispagestyle{empty}
\section*{Introduction}
Irreducible symplectic varieties are the singular hyper-K\"ahler factors in the Beauville--Bogomolov decomposition for compact K\"ahler klt spaces with numerically trivial canonical class~\cite{HoringPeternell2019AlgebraicIntegrability,BakkerGuenanciaLehn2022AlgebraicApproximationDecompositionTheorem}. Smooth examples are notoriously hard to construct, but since the recent development of the singular theory~\cite{BakkerLehn2022GlobalModuli}, new singular deformation types have appeared from several sources: moduli spaces of sheaves~\cite{PeregoRapagnetta2023ISVModuliSheaves}, compactified Lagrangian fibrations~\cite{LiuLiuXu2025IrreducibleSymplecticVarieties, Sacca2025CompactifyingLagrangianFibrations}, and terminalisations of finite quotients of \emph{smooth} symplectic manifolds, for instance~\cite{Menet2022DeformationClassesHKOrbifolds,BertiniEtAl2025TerminalizationsQuotients}. Here we follow the quotient route, but with a \emph{singular} input: Beauville--Mukai systems on K3 surfaces with del Pezzo quotients. This allows our quotients to retain a large second Betti number.

\smallskip
Let \(f\colon S\to T\) be a K3 double cover of a del Pezzo surface \(T\) of degree \(d\coloneqq K_T^2\), with covering involution \(\iota\in \mathrm{Aut}(S)\). Let \(H\coloneqq -f^\ast K_T\in \NS(S)\) be the pull-back of the anticanonical polarisation. The relative Jacobian fibration over the linear system \(|H|\), whose general member is a smooth curve of genus \(d+1\), admits a symplectic compactification given by the Beauville--Mukai system 
\[
  \pi\colon M_{\vv}(S,H)\to |H|\cong \mathbb{P}^{d+1},
\]
where \(M_{\vv}(S,H)\) is the good moduli space of one-dimensional \(H\)-semistable sheaves on \(S\) with Mukai vector \(\vv\coloneqq (0,H,-d)\), and \(\pi\) is the Fitting support morphism~\cite{Mukai1984SymplecticStructureModuliSheaves, Beauville1991SystmesHamiltoniensIntgrablesK3}.

\smallskip
A general member \(C\) of the \(\iota\)-invariant hyperplane \(f^\ast|-K_T|\subset |H|\) is a ramified double cover of a smooth genus-one curve in \(|-K_T|\). On the corresponding Jacobian fibre \(\Pic^0(C)\) there is a natural Prym involution
\(\mathscr{L}\mapsto \iota^\ast\mathscr{L}^{\vee}\). It is realised globally as a symplectic involution by combining the covering involution with a twisted derived duality:
\[
  \tau\colon M_\vv(S, H)\to M_\vv(S, H), \qquad
  \mathscr{F} \mapsto \mathbf{R}\mathcal{H}om(\iota^\ast\mathscr{F},\mathscr{O}_S(-H))[1].
\]
Variants of this involution, and especially its fixed locus, have appeared in constructions of relative compactified Prym varieties~\cite{MarkushevichTikhomirov2007Prymian, ArbarelloEtAl2015RelativePrymEnriques, Matteini2016SingularPrymFibration, BrakkeeEtAl2026RelativePrym}. In this paper we instead study terminal models of the quotient---see Theorem~\ref{thm:terminalisation-of-X}:

\begin{thmx}\label{thm:intro-main-examples}
  Let \(f\colon S\to T\) be a very general K3-del Pezzo double cover, where \(T\) has degree \(d\), with \(1\leq d\leq 9\). There is a unique \(\mathbb{Q}\)-factorial terminalisation 
  \[
    \widetilde{X}\to X\coloneqq M_{\vv}(S,H)/\langle \tau \rangle,
  \]
  where \(\widetilde{X}\) is an irreducible symplectic variety of dimension \(2d+2\) with \(b_2(\widetilde{X})=15+d\) and \(\rho(\widetilde{X})=3\). It carries a Lagrangian fibration to a weighted projective space 
  \[
    \widetilde{X}\to |H|/\langle \iota \rangle \cong \mathbb{P}(1^{d+1}, 2).
  \]
\end{thmx}

These terminal models are singular: they contain at least \(2^{2d+2}\) isolated singularities lying over the singular point of \(\mathbb{P}(1^{d+1}, 2)\). Thus we obtain singular irreducible symplectic varieties in every even dimension between \(4\) and \(20\), with second Betti numbers ranging from \(16\) to \(24\). The case \(d=1\) lies in the known deformation class of orbifolds of Nikulin type~\cite{CamereGarbagnatiKapustkaKapustka2024ProjectiveOrbifoldsNikulinType}, whereas for \(d\geq 2\) the examples are new deformation types. In particular, for \(2\leq d\leq 7\) they fill, to the best of the author's knowledge, the previously empty range \(17\leq b_2\leq 22\) among known irreducible symplectic varieties of dimension at least \(4\). At the upper end, \(d=9\) gives a \(20\)-dimensional example with \(b_2=24\), reaching the extremal value realised by O'Grady's smooth tenfold~\cite{OGrady1999DesingularizedModuliK3} and by the forthcoming \(20\)-dimensional example of Lin--Yamagishi~\cite{YamagishiLinInPreparationDualityInvolution}; it also lies in the same large-\(b_2\) range as the recent \(42\)-dimensional compactified Jacobian example of Liu--Liu--Xu~\cite{LiuLiuXu2025IrreducibleSymplecticVarieties}, for which \(b_2\geq 24\); see Remark~\ref{rmk:novelty-comparison}.

\smallskip
From a period-theoretic perspective, the Mukai morphism (see \textsection\ref{subsec:2nd-cohomology-moduli-sheaves}) provides a natural identification of rational transcendental Hodge structures
\[
  T(S)_{\mathbb{Q}}\cong T(\widetilde X)_{\mathbb{Q}}.
\]
Thus the construction gives very general points in the corresponding rank-three lattice-polarised Noether--Lefschetz locus in the moduli space. The main technical point is to control the algebraic part through chamber-to-wall contractions, quotients and terminalisations. 

\subsection*{Special features of the del Pezzo setup}

The same construction can be attempted for an arbitrary K3 surface \(S\) with an anti-symplectic involution \(\iota\) and an \(\iota\)-invariant linear system, in the generality considered in~\cite{BrakkeeEtAl2026RelativePrym}. Our choice of del Pezzo quotients and of the anticanonical linear system is motivated by two special features.

\smallskip
First, \(|-K_T|\) is a moving linear system whose general member is a smooth genus-one curve. After pull-back, it gives a family of double covers of genus-one curves, and hence a distinguished relative Prym component in the fixed locus of \(\tau\), of codimension \(2\); see Proposition~\ref{prop:fixed-locus-tau-delPezzo}. Its image in the quotient is therefore, generically, a transverse \(A_1\)-singularity. In particular the quotient has canonical but non-terminal singularities along it, and any \(\mathbb{Q}\)-factorial terminalisation extracts a divisor, as in Proposition~\ref{prop:geometry-X} and Theorem~\ref{thm:terminalisation-of-X}. Without such a codimension-\(2\) fixed component, the corresponding symplectic quotient is usually already terminal; see for instance~\cite[\textsection10.3]{ArbarelloEtAl2015RelativePrymEnriques}.

\smallskip
Second, the Prym involution acts trivially on the second cohomology, see Theorem~\ref{thm:trivial-action-cohomology}:

\begin{thmx}\label{thm:intro-trivial-action-cohomology}
  If \(d\geq 2\), then the Prym involution \(\tau\) acts trivially on \(H^2(M_\vv(S,H),\mathbb{Z})\).
\end{thmx}

Thus no degree-\(2\) classes are lost by passing to the quotient by \(\tau\). Within the realm of moduli spaces of sheaves on K3 surfaces, the phenomenon is genuinely \emph{singular}. Indeed, for smooth hyper-K\"ahler manifolds of \(\mathrm{K3}^{[n]}\)-type the natural representation on second cohomology is faithful~\cite[Lemma 1.2]{Mongardi2013NaturalDeformations}; a non-trivial symplectic involution therefore has a proper invariant subspace in \(H^2\), which translates into a drop in \(b_2\) for the quotient; see Remark~\ref{rmk:trivial-action-cohomology-K3n}.

\smallskip
Except in the wall-free cases \(d=1\) and \(d=9\), where there is essentially no variation of stability for our Mukai vector, the construction is tied to \(\vv\)-walls in \(\Amp(S)\), and hence to singular wall moduli spaces. The effect of the twisted duality on stability is described in Lemma~\ref{lem:duality-stability-reflection-delPezzo} and Corollary~\ref{cor:duality-reflection-delPezzo}: it typically exchanges chamber models adjacent across the wall face containing the ray \(\mathbb{R}_{\geq 0}H\), and only on this face does it become an automorphism of a single moduli space; see also Remark~\ref{rmk:forced-wall-prym-invariants}. Hence we are led to moduli spaces \(M_\vv(S,H)\) with non-generic polarisation, and sometimes non-primitive Mukai vector. Their global and local geometry is controlled below by chamber-to-wall partial resolutions and by Ext-quiver Kuranishi models, respectively; see \textsection\ref{subsec:Kuranishi-local-model} and Propositions~\ref{prop:polystable-strata-delPezzo} and~\ref{prop:geometry-Mv-delPezzo}. 

\subsection*{Relation to previous work}

As already pointed out, constructing irreducible symplectic varieties as terminal models of finite symplectic quotients of hyper-K\"ahler manifolds is not new, and has recently proved fruitful; see for instance~\cite{Menet2015BeauvilleBogomolovLattice, Menet2018IntegralCohomologyQuotients, Menet2022DeformationClassesHKOrbifolds, BertiniEtAl2025TerminalizationsQuotients,Mazzon2026TerminalizationsQuotientsFanoLines,DontenBuryEtAl2026SymplecticFourfold}. A special feature of our construction is that we start with a \emph{singular model}---which allows phenomena like Theorem~\ref{thm:intro-trivial-action-cohomology}---and show it develops strictly canonical singularities after a finite quotient.

The Prym involution \(\tau\) also appears in earlier constructions. Its fixed locus was first studied in this context by Markushevich--Tikhomirov~\cite{MarkushevichTikhomirov2007Prymian}, and later used for constructing several compact Lagrangian fibrations with Prym fibres~\cite{ArbarelloEtAl2015RelativePrymEnriques, Matteini2016SingularPrymFibration, BrakkeeEtAl2026RelativePrym}. 

The present paper should be viewed as a compact, Beauville--Mukai counterpart to the construction of symplectic resolutions of quotients of Hitchin systems in~\cite{AbuafCarini2024SemistableHiggs}. In the genus-\(3\) hyperelliptic case treated there, the local surface \(T^\ast C\) is the Higgs-side analogue of the hyperelliptic K3 surface \(f\colon S \to \mathbb{P}^1\times \mathbb{P}^1\) considered here. Via the spectral correspondence, Higgs bundles on \(C\) are one-dimensional sheaves on spectral curves in \(T^\ast C\). Under this dictionary, the autoequivalence of \(T^\ast C\) corresponding to \(\tau\) becomes the dualising--hyperelliptic involution of~\cite{AbuafCarini2024SemistableHiggs}; compare with~\cite[\textsection3.2]{Franco2022DegenerationLagrangiansPrymianIntegrableSystems}.  

\subsection*{Outline} 
In Section~\ref{sec:zoo-symplectic-varieties} we recall the various notions of singular symplectic varieties used throughout the paper, together with the basic extension and pull-back properties of reflexive differentials. In Section~\ref{sec:local-structure-moduli-sheaves} we collect the background on moduli spaces of sheaves on K3 surfaces: global geometry, second cohomology, local Kuranishi models, and the stratification by polystable type.

In Section~\ref{sec:BM-delPezzo-double-covers} we introduce K3-del Pezzo double covers and the Beauville--Mukai systems associated with the pull-back of the anticanonical linear system, which are our main object of study. We analyse the non-integral members of the support linear system, describe the minimal polystable strata, and compute the basic invariants of \(M_{\vv}(S,H)\) by means of chamber-to-wall partial resolutions.

In Section~\ref{sec:prym-involution-delPezzo} we introduce the Prym involution, describe its fixed locus, and prove that, for \(d\geq 2\), it acts trivially on \(H^2(M_{\vv}(S,H),\mathbb{Z})\). Finally, in Section~\ref{sec:terminalisation-quotient} we consider the quotient \(X\coloneqq M_{\vv}(S,H)/\langle \tau\rangle\), describe its singularities, and compute the invariants of any \(\mathbb{Q}\)-factorial terminalisation.

\subsection*{Notation and conventions} 
Throughout we work over the field \(\mathbb{C}\) of complex numbers. By \emph{variety} we mean an integral, separated scheme of finite type over \(\mathbb{C}\). We often identify a variety \(X\) with the associated complex analytic space \(X^{\mathrm{an}}\). Thus analytic and topological notions, such as manifold, analytic neighbourhood, fundamental group and singular cohomology, always refer to \(X^{\mathrm{an}}\). 

\smallskip
Given a variety \(X\), we denote by \(X^{\mathrm{reg}}\) and \(X^{\mathrm{sing}}\) the open and closed subsets given by the smooth and singular loci (with their reduced structures), respectively. When \(X\) is normal, and \(j\colon X^{\mathrm{reg}}\to X\) denotes the open immersion of its smooth locus, the sheaf of \emph{reflexive differential \(p\)-forms} is defined to be
\[ \Omega_X^{[p]}\coloneqq j_\ast \Omega_{X^{\mathrm{reg}}}^p \cong (\bwed^p \Omega_X)^{\vee \vee}, \quad 0\leq p\leq \dim X, \]
so that \(H^0(X, \Omega_X^{[p]})\cong H^0(X, \Omega^p_{X^{\mathrm{reg}}})\). In other words, a reflexive \(p\)-form is nothing but a regular \(p\)-form on the smooth locus.

\smallskip
We write \(\Pic(X)\) and \(\Cl(X)\) for the Picard group and the Weil divisor class group of a normal variety \(X\), respectively. For a normal projective variety \(X\), we write \(\NS(X)\) for the N\'eron--Severi group and \(\Amp(X)\subset \NS(X)_{\mathbb{R}}\) for the ample cone. We also write \(b_2(X)\coloneqq \rk H^2(X,\mathbb{Z})\) and \(\rho(X)\coloneqq \rk \NS(X)\).

For a morphism \(\phi\colon X\to Y\) between normal projective varieties, we write \(N^1(X/Y)\) and \(N_1(X/Y)\) for the real relative numerical divisor and curve spaces; thus \(N_1(X/Y)\) is generated by classes of curves contracted by \(\phi\), and \(N^1(X/Y)\) is its dual via intersection. We write \(\rho(X/Y)\coloneqq \dim_{\mathbb{R}} N^1(X/Y)\).

\smallskip
Given a normal variety \(X\), a \emph{birational contraction} \(\phi \colon X\to Y\) is a proper birational morphism onto a normal variety \(Y\). If \(Y\) has canonical singularities, we say such a \(\phi\) is a (\(\mathbb{Q}\)-\emph{factorial}) \emph{terminalisation} if it is crepant and \(X\) is (\(\mathbb{Q}\)-factorial and) terminal. In other words, terminalisations---or terminal models in the language of~\cite[Cor.\@ 1.4.3]{BirkarEtAl2010MinimalModelsLogGeneralType}, where existence is shown under the projective assumption---can be regarded as \emph{partial crepant resolutions}. 

\subsection*{Acknowledgments} 
The author is supported by the ERC Synergy Grant HyperK, Grant agreement ID 854361. He is grateful to Federico Bongiorno, Alessio Bottini, Luigi Martinelli, Hsueh-Yung Lin and Ryo Yamagishi for helpful discussions. 

\addtocontents{toc}{\protect\setcounter{tocdepth}{1}}
\tableofcontents

\section{Symplectic varieties}
\label{sec:zoo-symplectic-varieties}

We recall the main notions of singular symplectic varieties used in the literature, and the relations between them. For simplicity, we work in the projective setting, although the corresponding compact K\"ahler analytic notions and results are also available in the literature. 

\begin{definition}
  \label{def:zoo-symplectic-varieties}
  The following notions are listed in decreasing order of generality:
  \begin{enumerate}
    \item\cite[Def.\@ 1.1]{Beauville2000SymplecticSingularities} A \emph{symplectic variety} is a normal variety \(X\) endowed with a closed reflexive 2-form \(\omega \in H^0(X, \Omega_X^{[2]})\) which is non-degenerate on the smooth locus \(X^{\mathrm{reg}}\) and extends to a regular form \(\widetilde{\omega}\in H^0(\widetilde{X}, \Omega_{\widetilde{X}}^2)\) on any resolution of singularities \(\phi \colon \widetilde{X}\to X\).
    \item\cite[Def.\@ 3.1]{BakkerLehn2022GlobalModuli} A \emph{primitive symplectic variety} is a projective symplectic variety \((X, \omega)\) such that \(H^0(X, \Omega_X^{[2]})\) is spanned by \(\omega\) and \(H^1(X, \mathscr{O}_X)=0\).
    \item\cite[Def.\@ 1.6]{Matsushita2015BaseManifolds} A \emph{cohomologically irreducible symplectic variety}\footnote{Note that Matsushita further assumes that \(X\) is \(\mathbb{Q}\)-factorial and terminal.} is a projective symplectic variety \((X, \omega)\) whose algebra of reflexive differentials is generated by \(\omega\). In other words, there is an isomorphism of graded \(\mathbb{C}\)-algebras
          \[\bigoplus_{p=0}^{\dim X}H^0(X, \Omega_X^{[p]})\cong \mathbb{C}[\omega].\]
    \item\cite[Def.\@ 8.16.2]{GrebKebekusPeternell2016SingularSpaces} An \emph{irreducible symplectic variety} is a projective symplectic variety \((X, \omega)\) such that, for any finite quasi-\'etale morphism \(\phi \colon Y\to X\), the pair \((Y, \phi^{[\ast]} \omega)\) is cohomologically irreducible. In other words, the algebra of reflexive differentials of \(Y\) is generated by \(\phi^{[\ast]} \omega\), namely there is an isomorphism of graded \(\mathbb{C}\)-algebras
          \[ \bigoplus_{p=0}^{\dim Y}H^0(Y, \Omega_Y^{[p]})\cong \mathbb{C}[\phi^{[\ast]} \omega]. \]
  \end{enumerate}
\end{definition}

Symplectic varieties have rational Gorenstein singularities, hence canonical singularities. Thus the framework of~\cite{KebekusSchnell2021ExtendingHolomorphicForms}---or its algebraic counterpart~\cite{GrebEtAl2011DifferentialForms}---applies. In particular, for any resolution of singularities \(\phi \colon \widetilde{X}\to X\) of a symplectic variety, there is an isomorphism
\begin{equation*}
  \phi_\ast \Omega_{\widetilde{X}}^p \iso \Omega_X^{[p]}
\end{equation*}
for all \(0\leq p\leq \dim X\). Equivalently, reflexive differentials on \(X\) extend uniquely to any resolution of singularities~\cite[Cor.\@ 1.8]{KebekusSchnell2021ExtendingHolomorphicForms}. On global sections, this gives \(H^0(\widetilde{X}, \Omega_{\widetilde{X}}^p)\cong H^0(X, \Omega_X^{[p]})\).

\smallskip
More generally, for any morphism \(\phi \colon X\to Y\) between symplectic varieties, there is a functorial pull-back morphism
\begin{equation*}
  \phi^{[\ast]}\colon \phi^\ast \Omega_Y^{[p]}\to \Omega_X^{[p]}
\end{equation*}
for all \(0\leq p\leq \dim Y\)~\cite[Thm.\@ 1.11]{KebekusSchnell2021ExtendingHolomorphicForms}. If moreover \(\phi\) is a birational contraction, for instance a terminalisation, then taking a common resolution gives an isomorphism \(H^0(Y, \Omega_Y^{[p]})\cong H^0(X, \Omega_X^{[p]})\). See~\cite[\textsection2,3]{BakkerLehn2022GlobalModuli} for a review of the Hodge theory of rational and symplectic singularities.

\smallskip
Primitive symplectic varieties form the broadest class for which Bakker--Lehn develop a moduli and deformation theory analogous to the smooth case, in particular one suited to locally trivial analytic deformations~\cite{BakkerLehn2022GlobalModuli}. Irreducible symplectic varieties, on the other hand, are the \emph{holomorphic symplectic factors} in the projective singular Beauville--Bogomolov decomposition theorem~\cite[Thm.\@ 1.5]{HoringPeternell2019AlgebraicIntegrability}; see also its compact K\"ahler extension~\cite[Thm.\@ A]{BakkerGuenanciaLehn2022AlgebraicApproximationDecompositionTheorem}. See~\cite{GrebGuenanciaKebekus2019KltTrivialCanonical, BakkerLehn2022GlobalModuli} for a fuller account of the relations between these notions, including examples and counterexamples.

\smallskip
Let \(X\) be an irreducible symplectic variety. Then \(\pi_1(X)=\{1\}\)~\cite[\textsection13]{GrebGuenanciaKebekus2019KltTrivialCanonical}, and hence \(H^2(X, \mathbb{Z})\) is torsion-free. Moreover, \(H^2(X,\mathbb{Z})\) carries a pure weight-two Hodge structure~\cite[Lemma 2.1]{BakkerLehn2022GlobalModuli}, and is endowed with a non-degenerate integral quadratic form, the \emph{Beauville--Bogomolov--Fujiki form}~\cite[\textsection5.1]{BakkerLehn2022GlobalModuli}. With respect to this form, we write
\[
  T(X)\coloneqq \NS(X)^\perp\subset H^2(X,\mathbb{Z})
\]
for the integral \emph{transcendental lattice} of $X$; equivalently, it is the minimal primitive sublattice of $H^2(X, \mathbb{Z})$ with $H^{2,0}(X)\cong  H^0(X, \Omega_X^{[2]})\subseteq T(X)_\mathbb{C}$. 

Note that the fundamental group \(\pi_1(X^{\mathrm{reg}})\) of the smooth locus may be infinite, although its profinite completion is finite. Since finite quasi-\'etale covers of \(X\) correspond to finite \'etale covers of \(X^{\mathrm{reg}}\), Definition~\ref{def:zoo-symplectic-varieties}.\(\mathrm{(iv)}\) must take all such covers into account.

\smallskip
The following are the permanence criteria used below.

\begin{proposition}
  \label{prop:criteria-ISV}
  We refer to the classes \(\mathrm{(i)-}\mathrm{(iv)}\) in \textnormal{Definition~\ref{def:zoo-symplectic-varieties}}.
  \begin{enumerate}
    \item The classes \(\mathrm{(i)-}\mathrm{(iv)}\) are closed under terminalisations and under finite quotients preserving the symplectic form.
    \item The classes \(\mathrm{(i)-}\mathrm{(iii)}\) are closed under birational contractions.
    \item A cohomologically irreducible symplectic variety \((X, \omega)\) with \(\pi_1(X^{\mathrm{reg}})=\{1\}\) is an irreducible symplectic variety.
  \end{enumerate}
\end{proposition}

\begin{proof}
  The assertions in \(\mathrm{(i)}\) and \(\mathrm{(ii)}\), where not immediate from the definitions and the preceding discussion, follow from~\cite[Ex.\@ 3.2]{BakkerLehn2022GlobalModuli},~\cite[Prop.\@ 3.17]{BertiniEtAl2025TerminalizationsQuotients},~\cite[Prop.\@ 6.9]{GrebKebekusPeternell2016SingularSpaces}, and~\cite[Prop.\@ 5.8]{Kebekus2013PullBackReflexiveForms}. Assertion \(\mathrm{(iii)}\) follows from the holonomy criterion~\cite[Prop.\@ F]{GrebGuenanciaKebekus2019KltTrivialCanonical}.
\end{proof}

\begin{remark}[Irreducibility is not birational]
  \label{rmk:cohomologically-irred-not-irred}
  Irreducible symplectic varieties are \emph{not} closed under birational contractions. For instance, if \(S\) is a smooth K3 surface, then \(X\coloneqq \operatorname{Sym}^2 S\) is a cohomologically irreducible symplectic variety, but it admits the finite quasi-\'etale cover \(S\times S\to X\), which is not cohomologically irreducible. Here \(\pi_1(X^{\mathrm{reg}})\cong \mathfrak{S}_2\) is non-trivial. Nevertheless, the crepant resolution \(\operatorname{Hilb}^2 S\to X\) is a birational contraction from a smooth irreducible holomorphic symplectic variety.
\end{remark}

\section{Generalities on moduli spaces of sheaves on K3 surfaces}
\label{sec:local-structure-moduli-sheaves}

This section collects the background on moduli spaces of sheaves on K3 surfaces used throughout the paper. We recall their global geometry, especially in the singular non-primitive cases, the Mukai-theoretic description of second cohomology, and the local Kuranishi model at a polystable sheaf, together with its interpretation as a symplectic reduction and the induced stratification by polystable type.

\subsection{Global geometry of moduli spaces of sheaves}

Let \(S\) be a projective K3 surface. We use the standard \emph{Mukai lattice}
\(
  \widetilde H(S,\mathbb{Z})\coloneqq H^\ast(S,\mathbb{Z})
\)
with its Mukai pairing and weight-two Hodge structure, following for instance~\cite[Ch.\@ 6]{HuybrechtsLehn2010GeometryModuliSheaves} and~\cite[\textsection1.1]{PeregoRapagnetta2023ISVModuliSheaves}. Its algebraic part is
\[
  \widetilde H_{\mathrm{alg}}(S,\mathbb{Z})\coloneqq H^0(S,\mathbb{Z})\oplus \NS(S)\oplus H^4(S,\mathbb{Z}),
\]
the integral \((1,1)\)-part of \(\widetilde H(S,\mathbb{Z})\). For a coherent sheaf \(\mathscr{F}\) on \(S\), we write
\[
  \vv(\mathscr{F})\coloneqq \operatorname{ch}(\mathscr{F})\sqrt{\operatorname{td}_S}\in \widetilde H_{\mathrm{alg}}(S,\mathbb{Z})
\]
for its \emph{Mukai vector}~\cite[Def.\@ 6.1.4]{HuybrechtsLehn2010GeometryModuliSheaves}. For a fixed Mukai vector \(\vv\in \widetilde H_{\mathrm{alg}}(S,\mathbb{Z})\) and a polarisation \(H\in \Amp(S)\), we write \(M_\vv(S, H)\) for the good moduli space of \(H\)-Gieseker semistable sheaves on \(S\) with Mukai vector \(\vv\)~\cite{Simpson1994ModuliRepresentationsFundamentalGroupI,HuybrechtsLehn2010GeometryModuliSheaves}. Its closed points parametrise \(H\)-polystable sheaves. We write \(M^s_\vv(S,H)\subset M_\vv(S,H)\) for the open subset parametrising isomorphism classes of stable sheaves. 

\smallskip
The dependence of \(H\)-Gieseker semistability on the polarisation is governed by the standard wall-and-chamber decomposition of the ample cone. The \(\vv\)-\emph{walls} are locally finite rational hyperplanes in \(\Amp(S)\); the connected components of their complement are the \(\vv\)-\emph{chambers}, and a polarisation is \(\vv\)-\emph{generic} if it lies in a chamber. For \(H\) varying in a fixed chamber, the set of \(H\)-semistable sheaves of Mukai vector \(\vv\), and hence the moduli space \(M_\vv(S,H)\), remains unchanged; crossing a wall is where strictly semistable objects may appear or disappear. If \(\vv\) is \emph{primitive}, namely indivisible in \(\widetilde H_{\mathrm{alg}}(S,\mathbb{Z})\), and \(H\) is \(\vv\)-generic, then semistability coincides with stability, so \(M^s_\vv(S,H)=M_\vv(S,H)\). See~\cite[\textsection2.1]{PeregoRapagnetta2023ISVModuliSheaves} for a precise account of this wall-and-chamber decomposition and for the explicit wall equations.

\smallskip
We record the global results on \(M_\vv(S,H)\) used below. Throughout, we assume that \(\vv \in \widetilde H_{\mathrm{alg}}(S,\mathbb{Z})\) satisfies the standard \emph{positivity} condition listed in~\cite[\textsection 1.1]{PeregoRapagnetta2023ISVModuliSheaves}, or equivalently~\cite[Condition \((\ast)\), \textsection 2.3]{KaledinLehnSorger2006SingularSymplecticModuli}, which ensures non-emptiness of the corresponding moduli spaces.

\begin{theorem}\label{thm:global-geometry-Mv-delPezzo}
  Write \(\vv\coloneqq m\vv_0\in \widetilde H_{\mathrm{alg}}(S,\mathbb{Z})\), where \(m\geq 1\) and \(\vv_0\in \widetilde H_{\mathrm{alg}}(S,\mathbb{Z})\) is primitive with \(\vv_0^2\geq 0\). Let \(H\in \Amp(S)\) be a \(\vv\)-generic polarisation. Then:
  \begin{enumerate}
    \item If \(\vv_0^2=0\), then \(M_{\vv_0}(S, H)\) is a K3 surface, and \(M_{\vv}(S, H)\cong \operatorname{Sym}^m M_{\vv_0}(S, H)\) is a cohomologically irreducible symplectic variety.
    \item If \(m=1\) and \(\vv_0^2>0\), then \(M_{\vv}(S, H)\) is a hyper-K\"ahler manifold of \(\mathrm{K3}^{[n]}\)-type, where \(2n=\vv_0^2+2\).
    \item If \(m=2\) and \(\vv_0^2=2\), then \(M_{\vv}(S, H)\) is a \(\mathbb{Q}\)-factorial irreducible symplectic variety. Its regular locus is the stable locus \(M^s_\vv(S,H)\), and it is simply connected. It admits a symplectic resolution of \(\mathrm{OG}10\)-type given by the blow-up of the reduced singular locus.
    \item If \(m=2\) and \(\vv_0^2>2\), or \(m\geq 3\) and \(\vv_0^2>0\), then \(M_{\vv}(S, H)\) is a locally factorial irreducible symplectic variety. Its regular locus is the stable locus \(M^s_\vv(S,H)\), and it is simply connected. It has terminal singularities and does not admit any symplectic resolution.
  \end{enumerate}
\end{theorem}

Moreover, it is shown in~\cite[Thm.\@ 1.7]{PeregoRapagnetta2023ISVModuliSheaves} that the (locally trivial) deformation class of each of the above symplectic varieties depends only on the pair \((m, \vv_0^2)\).

\begin{proof} The isotropic case is due to Mukai~\cite[Thm.\@ 1.4, Prop.\@ 3.13]{Mukai1987ModuliBundlesK3}; see also~\cite[Prop.\@ 4.7]{BottiniCarini2026SemiRigidStableSheaves} for the case \(m\geq 2\). The primitive case \(m=1\) follows from~\cite{Mukai1984SymplecticStructureModuliSheaves, OGrady1997WeightTwoHodgeStructureModuliSheaves, Yoshioka2000IrreducibilityModuliK3,Yoshioka2001ModuliAbelianSurfaces}. The case \(m=2\) and \(\vv_0^2=2\) is due to O'Grady~\cite{OGrady1999DesingularizedModuliK3}; see also~\cite[Thm.\@ 1.1]{LehnSorger2006SingularityOGrady},~\cite[Thm.\@ 1.1]{PeregoRapagnetta2014FactorialityModuliSheaves} and~\cite[Thm.\@ 1.10]{PeregoRapagnetta2023ISVModuliSheaves}. The remaining cases were studied by Kaledin--Lehn--Sorger~\cite{KaledinLehnSorger2006SingularSymplecticModuli}; see also~\cite[Thm.\@ 1.10]{PeregoRapagnetta2023ISVModuliSheaves}. The statements about the fundamental group of the regular locus are due to Perego--Rapagnetta~\cite[Thm.\@ 3.4]{PeregoRapagnetta2023ISVModuliSheaves}. 
\end{proof}

For non-generic polarisations, wall moduli spaces are subtler: there are no general results on irreducibility, reducedness or normality. In the pure one-dimensional situations considered below, however, they can be understood through chamber-to-wall birational morphisms from adjacent generic chambers. The required birational geometry framework is supplied by the Bayer--Macr\`i description of the MMP for moduli spaces of sheaves on K3 surfaces~\cite{BayerMacri2014MMPModuliSheaves}.

\subsection{Second cohomology and N\'eron--Severi lattice}\label{subsec:2nd-cohomology-moduli-sheaves}

Let \(\vv\coloneqq m\vv_0\), with \(\vv_0\) primitive and \(\vv_0^2>0\), and let \(H\in \Amp(S)\) be \(\vv\)-generic. By Theorem~\ref{thm:global-geometry-Mv-delPezzo}, \(M_\vv(S,H)\) is an irreducible symplectic variety, so its second cohomology carries the Hodge structure and Beauville--Bogomolov--Fujiki form recalled in Section~\ref{sec:zoo-symplectic-varieties}.

We recall the Mukai morphism comparing this second cohomology group with the Mukai lattice.  In the smooth primitive case, the construction and its Hodge-theoretic properties go back to Mukai, O'Grady and Yoshioka~\cite{Mukai1984SymplecticStructureModuliSheaves,OGrady1997WeightTwoHodgeStructureModuliSheaves,Yoshioka2000IrreducibilityModuliK3,Yoshioka2001ModuliAbelianSurfaces}. For the singular moduli spaces considered here we use the formulation of~\cite{PeregoRapagnetta2024SecondIntegralCohomology}, see also the references therein.

\begin{theorem}[{\cite[Thm.\@ 1.6]{PeregoRapagnetta2024SecondIntegralCohomology}}]\label{thm:2nd-cohomology-moduli-sheaves} 
  Let \(\vv\coloneqq m\vv_0\in \widetilde H_{\mathrm{alg}}(S,\mathbb{Z})\), where \(m\geq 1\), \(\vv_0\) is primitive and \(\vv_0^2>0\). For a \(\vv\)-generic polarisation \(H\in \Amp(S)\), there exists a Mukai morphism
  \[
    \lambda_{\vv}\colon \vv^\perp \to H^2(M_\vv(S,H), \mathbb{Z}), \qquad
    \vv^\perp\coloneqq \{\mathbf{w}\in \widetilde H(S,\mathbb{Z})\mid \mathbf{w}\cdot \vv=0\}.
  \]
  This morphism is an isomorphism of abelian groups and a Hodge isometry. In particular,
  \[
    b_2(M_\vv(S, H))=23, \quad \rho(M_\vv(S, H))=\rho(S)+1.
  \]
\end{theorem}

We fix the normalisation of \(\lambda_{\vv}\). Let \(j\colon M^s_\vv(S,H)\hookrightarrow M_\vv(S,H)\) be the inclusion of the stable locus. On \(S\times M^s_\vv(S,H)\) there exists a quasi-universal family \(\mathcal{E}\) of similitude \(\rho\) in the sense of~\cite[Def.\@ 4.6.1]{HuybrechtsLehn2010GeometryModuliSheaves}. The cohomological Le Potier construction~\cite[Def.\@ 4.2]{PeregoRapagnetta2024SecondIntegralCohomology}, a natural generalisation of the Le Potier morphism of~\cite[Def.\@ 8.1.1]{HuybrechtsLehn2010GeometryModuliSheaves}, assigns to any class \(\mathbf{w}\in \widetilde H(S,\mathbb{Z})\) a natural class 
\[\lambda_{\mathcal{E}}(\mathbf{w})\in H^2(M^s_\vv(S,H),\mathbb{Z}).\] 
After tensoring with \(\mathbb{Q}\), the restriction morphism
\[
  j^\ast \colon H^2(M_\vv(S,H),\mathbb{Q}) \to H^2(M^s_\vv(S,H),\mathbb{Q})
\]
normalises the canonical Mukai morphism by
\[
  j^\ast\lambda_{\vv}(\mathbf{w})=\frac{1}{\rho}\lambda_{\mathcal{E}}(\mathbf{w})\in H^2(M^s_\vv(S,H),\mathbb{Q}), \quad \mathbf{w}\in \vv^\perp,
\]
according to~\cite[Prop.\@ 4.4(2), Def.\@ 4.6(1)]{PeregoRapagnetta2024SecondIntegralCohomology}. 
This normalised class is independent of the choice of quasi-universal family when evaluated on \(\vv^\perp\), by~\cite[Lemma 4.3(3)]{PeregoRapagnetta2024SecondIntegralCohomology}.

\smallskip
For later use, we record the resulting explicit degree computation on projective spaces of extensions:

\begin{example}\label{ex:Mukai-Lepotier-extensions} Assume \(\vv=\vv_1+\vv_2\in \widetilde H_{\mathrm{alg}}(S,\mathbb{Z})\). Let \(\mathscr{F}\) and \(\mathscr{G}\) be stable sheaves on \((S,H)\) with \(\vv(\mathscr{F})=\vv_1\) and \(\vv(\mathscr{G})=\vv_2\). Suppose that all non-split extensions
  \[  0\to \mathscr{G}\to \mathscr{E}\to \mathscr{F}\to 0 \]
are stable. Let \(V\coloneqq \Ext^1(\mathscr{F}, \mathscr{G})\). The projective space \(\mathbb{P}V\) parametrising such extensions has a natural modular morphism \(f_{\mathcal{E}} \colon \mathbb{P}V\to M_\vv(S, H)\), induced by the universal extension \(\mathcal{E}\) on \(S\times \mathbb{P}V\) 
 \[  0\to p^\ast \mathscr{G}\otimes q^\ast \mathscr{O}_{\mathbb{P}V}(1) \to \mathcal{E} \to p^\ast \mathscr{F} \to 0, \]
 where \(p\) and \(q\) are the natural projections onto \(S\) and \(\mathbb{P}V\), respectively. For any \(\mathbf{w}\in \vv^\perp\), it follows from~\cite[Lemma 4.3, Prop.\@ 4.4]{PeregoRapagnetta2024SecondIntegralCohomology} that the restriction of \(\lambda_{\vv}(\mathbf{w})\) to \(\mathbb{P}V\) is
 \[ f_{\mathcal{E}}^\ast \lambda_{\vv}(\mathbf{w})= (\mathbf{w}\cdot \vv_2) \ c_1(\mathscr{O}_{\mathbb{P}V}(1))=-(\mathbf{w}\cdot \vv_1)\ c_1(\mathscr{O}_{\mathbb{P}V}(1))\in H^2(\mathbb{P}V, \mathbb{Z}). \]
\end{example}

\subsection{Local structure: Kuranishi models, symplectic reduction and Ext-quivers}
\label{subsec:Kuranishi-local-model}
Let \(\mathscr{F}\) be a polystable sheaf on \((S,H)\) with Mukai vector \(\vv\coloneqq \vv(\mathscr{F}) \in \widetilde H_{\mathrm{alg}}(S,\mathbb{Z})\). We describe the local structure of \(M_{\vv}(S, H)\) at the point \([\mathscr{F}]\). Deformation theory in the sense of Schlessinger--Rim~\cite{Schlessinger1968FunctorsArtinRings,Rim1972FormalDeformationTheory} provides a versal deformation space \(\mathrm{Def}_{\mathscr{F}}\coloneqq \mathrm{Spf}(R)\), where \(R\) is a complete Noetherian local \(\mathbb{C}\)-algebra. The reductive automorphism group \(G\coloneqq \mathrm{Aut}(\mathscr{F})\) acts naturally on the deformation functor, and hence \(\mathrm{Def}_{\mathscr{F}}\) comes with a natural \(G\)-action. Passing to good moduli spaces identifies the completed local ring of the moduli space with the invariant ring
\begin{equation}
  \label{eq:moduli-space-formally-SpfR}
  \hat{\mathscr{O}}_{M_{\vv}(S,H), [\mathscr{F}]}\cong R^G.
\end{equation}

By Kuranishi theory and the formality result established in~\cite{BudurZhang2019FormalityK3, BandieraManettiMeazzini2021FormalityMinimalSurfaces}, \(\mathrm{Def}_\mathscr{F}\) has an explicit algebraic model as the zero locus of the \(G\)-equivariant Yoneda square map
\begin{equation}
  \label{eq:Yoneda-square-mu}
  \mu \colon \Ext^1(\mathscr{F}, \mathscr{F})\to \Ext^2_0(\mathscr{F}, \mathscr{F}), \quad e\mapsto e\cup e.
\end{equation}
In other words, the germ of the quadratic cone \(\mu^{-1}(0)\subset \Ext^1(\mathscr{F}, \mathscr{F})\) at the origin is a representative of \(\mathrm{Def}_\mathscr{F}\), namely there is a \(G\)-equivariant isomorphism
\begin{equation*}
  \hat{\mathscr{O}}_{\mu^{-1}(0), 0}\cong R.
\end{equation*}
Consequently, after~\eqref{eq:moduli-space-formally-SpfR}, the affine GIT quotient \(\mu^{-1}(0)\gitq G\) provides a local model for the moduli space \(M_{\vv}(S,H)\) at the point \([\mathscr{F}]\)~\cite[\textsection4]{ArbarelloSacca2018K3QuiverSingularities}.

\smallskip
Following~\cite{KaledinLehnSorger2006SingularSymplecticModuli}, the quotient \(\mu^{-1}(0)\gitq G\) admits a natural interpretation as a symplectic reduction. More precisely, write the polystable decomposition of \(\mathscr{F}\) as
\begin{equation*}
  \mathscr{F}= \bigoplus_{i=1}^k \mathscr{F}_i \otimes U_i,
\end{equation*}
where the \(\mathscr{F}_i\) are pairwise non-isomorphic stable sheaves with Mukai vectors \(\vv_i\coloneqq \vv(\mathscr{F}_i)\in \widetilde H_{\mathrm{alg}}(S,\mathbb{Z})\), and the \(U_i\) are complex vector spaces of dimension \(n_i\in \mathbb{N}\). In analogy with quiver representations~\cite{LeBruynProcesi1990SemisimpleQuivers, CrawleyBoevey2001MomentMapQuivers}, the \emph{polystable type} of \(\mathscr{F}\) is
\begin{equation*}
  \mathfrak{t} \coloneqq (n_1, \vv_1; \dots ; n_k, \vv_k).
\end{equation*}
Since \(\mathscr{F}\) is \(H\)-polystable, the stable summands \(\mathscr{F}_i\) all have the same reduced Hilbert polynomial as \(\mathscr{F}\). In particular, for all \(i=1,\dots,k\), one has
\[
  p_H(\vv_i)=p_H(\vv).
\]

\smallskip
For each \(\ell\geq 0\) there are \(G\)-equivariant decompositions
\begin{equation*}
  \Ext^\ell (\mathscr{F}, \mathscr{F})\cong \bigoplus_{i, j} \Ext^\ell(\mathscr{F}_i, \mathscr{F}_j) \otimes \Hom(U_i, U_j),
\end{equation*}
where \(G\cong \prod_{i=1}^k \mathrm{Aut}(U_i) \cong \prod_{i=1}^k \mathrm{GL}_{n_i}\) acts on the second factor of each direct summand. The central subgroup \(\mathbb{G}_m\subset G\) of scalar automorphisms acts trivially, so we may also replace \(G\) by the quotient \(\mathbb{P}G\coloneqq G/\mathbb{G}_m\).

For each pair \((i, j)\) Serre duality provides a skew-symmetric perfect pairing
\begin{equation*}
  \Ext^1(\mathscr{F}_i, \mathscr{F}_j)\otimes \Ext^1(\mathscr{F}_j, \mathscr{F}_i)\to H^2(S, \mathscr{O}_S)\cong \mathbb{C}, \quad e_1\otimes e_2 \mapsto \mathrm{tr}(e_1\cup e_2).
\end{equation*}
Together with the standard symmetric trace pairing on \(\Hom(U_i, U_j)\otimes \Hom(U_j, U_i)\), these induce a natural \(G\)-invariant symplectic form on \(\Ext^1(\mathscr{F}, \mathscr{F})\). Moreover, stability implies \(\Ext^2(\mathscr{F}_i, \mathscr{F}_j)\cong \Hom(\mathscr{F}_j, \mathscr{F}_i)^\vee =0\) for \(i\neq j\), so that
\begin{equation*}
  \Ext^2_0(\mathscr{F}, \mathscr{F})\cong \ker \left( \bigoplus_{i=1}^k \mathrm{End}(U_i)\overset{\mathrm{tr}}{\to} \mathbb{C} \right),
\end{equation*}
which identifies canonically with \(\mathrm{Lie}(\mathbb{P}G)^\vee \). Under these identifications, the Yoneda square map~\eqref{eq:Yoneda-square-mu} coincides with the moment map for the symplectic \(\mathbb{P}G\)-action on the symplectic vector space \(\Ext^1(\mathscr{F}, \mathscr{F})\). In particular, the quotient
\[ \mu^{-1}(0)\gitq G \cong \Ext^1(\mathscr{F}, \mathscr{F})\sympq \mathbb{P}G \]
is the affine symplectic reduction of \(\Ext^1(\mathscr{F}, \mathscr{F})\) by \(\mathbb{P}G\).

\smallskip
Thus the analytic type of the moduli space is controlled by the Ext-quiver of the polystable decomposition. More precisely, the above linear algebraic data can be encoded by a quiver \(Q\), together with the dimension vector \(\mathbf{n}\coloneqq (n_1, \dots, n_k)\in \mathbb{N}^k\). The non-oriented Ext-graph underlying \(Q\) depends only on the polystable type \(\mathfrak{t}\) of \(\mathscr{F}\), and more precisely on the number \(k\geq 1\) of non-isomorphic stable summands, giving the vertices of \(Q\), and on the matrix
\begin{equation*}
  \dim \Ext^1(\mathscr{F}_i, \mathscr{F}_j)=\vv_i\cdot \vv_j + 2\delta_{ij},
\end{equation*}
whose entries give the number of edges between vertices \(i\) and \(j\) for \(i\neq j\), and \emph{twice} the number of loops for \(i=j\). After choosing an orientation of these edges, one obtains \(Q\). See~\cite{ArbarelloSacca2018K3QuiverSingularities, Toda2018ModuliStacksExtQuivers} for a precise description. With this notation, there is a natural isomorphism of symplectic \(G\)-representations
\begin{equation*}
  \Ext^1(\mathscr{F}, \mathscr{F})\cong T^\vee \mathrm{Rep}(Q, \mathbf{n})\cong \mathrm{Rep}(\overline{Q}, \mathbf{n})
\end{equation*}
where \(\overline{Q}\coloneqq Q\sqcup Q^{\mathrm{op}}\) denotes the doubled quiver. Consequently, the local model \(\mu^{-1}(0)\gitq G\) identifies with the affine \emph{Nakajima quiver variety}
\begin{equation*}
  \mathfrak{M}(Q, \mathbf{n})\coloneqq \mathrm{Rep}(\overline{Q}, \mathbf{n})\sympq \mathrm{GL}(\mathbf{n}).
\end{equation*}
This description allows one to analyse the local structure of \(M_{\vv}(S,H)\) near \([\mathscr{F}]\) using the literature on quiver representations. In particular, many local properties of the moduli space can be read off from the combinatorics of the associated quiver, using the structural results of~\cite{CrawleyBoevey2001MomentMapQuivers, CrawleyBoevey2003MarsdenWeinsteinNormality, BellamySchedler2021QuiverResolutions}.

\subsection{Stratification by polystable type}
\label{subsec:stratification-polystable-type}

The following result is the natural counterpart of the stratification by representation type of quiver varieties~\cite{LeBruynProcesi1990SemisimpleQuivers, CrawleyBoevey2001MomentMapQuivers} in the context of moduli spaces of sheaves on symplectic surfaces.

\begin{lemma}[\cite{BayerMacri2014MMPModuliSheaves, ArbarelloSacca2018K3QuiverSingularities, BrakkeeEtAl2026RelativePrym}]
  \label{lem:stratification-polystable-type}
  There is a finite stratification of \(M_{\vv}(S,H)\) by polystable type, whose strata are smooth and isosingular.
\end{lemma}

Recall that a locally closed subscheme \(Z\subset X\) is \emph{isosingular} if the singularity type of \(X\) is the same at all points of \(Z\). Equivalently, for any two closed points \(x, y\in Z\), there is an isomorphism of complete local \(\mathbb{C}\)-algebras \(\hat{\mathscr{O}}_{X, x}\cong \hat{\mathscr{O}}_{X, y}\)~\cite{Ephraim1978IsosingularLoci, ChiuHauser2022IsosingularLoci}.

\begin{proof}
  First note that there are only finitely many polystable types \(\mathfrak{t} \coloneqq (n_1, \vv_1;\dots; n_k, \vv_k)\) with \(\sum_{i=1}^k n_i \vv_i=\vv\), since the \(\vv_i\) must satisfy \(\vv_i^2\geq -2\) and \(p_H(\vv_i)=p_H(\vv)\), while the \(n_i\) are positive integers. For such a type \(\mathfrak{t}\), consider the direct-sum morphism
  \begin{equation}\label{eq:direct-sum-morphism}
    \prod_{i=1}^k M_{\vv_i}(S,H)\to M_{\vv}(S,H), \quad ([\mathscr{F}_1], \dots, [\mathscr{F}_k]) \mapsto [\mathscr{F}_1^{\oplus n_1}\oplus \cdots \oplus \mathscr{F}_k^{\oplus n_k}].
  \end{equation}
  Now for each \(i\neq j\) let \(\Delta_{ij}\subset \prod_{i=1}^k M_{\vv_i}(S,H)\) be either empty if \(\vv_i\neq \vv_j\) or the diagonal in \(M_{\vv_i}(S,H)\times M_{\vv_j}(S,H)\) if \(\vv_i=\vv_j\). We define \(M_{\vv}^\mathfrak{t} (S,H)\) to be the image under~\eqref{eq:direct-sum-morphism} of the smooth open subset
  \begin{equation*}
    U\coloneqq \prod_{i=1}^k M^s_{\vv_i}(S,H) \setminus \bigcup_{i\neq j} \Delta_{ij}.
  \end{equation*}
  Since~\eqref{eq:direct-sum-morphism} is finite, the image of \(U\) is locally closed in \(M_\vv(S,H)\). The ambiguity in the order of the stable summands is encoded by a finite group acting freely on \(U\), so \(M^\mathfrak{t}_{\vv}(S,H)\) is smooth. Finally, by the discussion above, the \'etale-local model of \(M_{\vv}(S,H)\) at a point \([\mathscr{F}]\) depends only on its polystable type. Thus it is the same for all points of \(M_{\vv}^\mathfrak{t}(S,H)\), and the stratum \(M_{\vv}^\mathfrak{t}(S,H)\) is isosingular.
\end{proof}

The stratum \(M_{\vv}^\mathfrak{t}(S,H)\) is non-empty if and only if the stable loci \(M^s_{\vv_i}(S,H)\) are non-empty for all \(i\). As in~\cite[\textsection4]{LeBruynProcesi1990SemisimpleQuivers}, there is a natural partial order on the set of polystable types \(\mathfrak{t}\) with \(\sum_{i=1}^k n_i \vv_i=\vv\), induced by elementary successor moves of the form
\begin{equation*}
  \begin{split}
    (n, \mathbf{w}_1+\mathbf{w}_2)  \leadsto (n, \mathbf{w}_1)+(n, \mathbf{w}_2), \\
    (n_1, \mathbf{w})+(n_2, \mathbf{w}) \leadsto (n_1+n_2, \mathbf{w}).
  \end{split}
\end{equation*}
In terms of this partial order, one has \(\mathfrak{t} \preceq \mathfrak{t}'\) if and only if the stratum \(M_{\vv}^{\mathfrak{t}'}(S,H)\) is contained in the closure of \(M_{\vv}^\mathfrak{t}(S,H)\).

\smallskip
The unique minimal stratum of type \((1,\vv)\) is precisely the stable locus \(M_{\vv}^s(S,H)\). If it is non-empty---for instance when \(\vv^2>0\) and \(H\) is \(\vv\)-generic~\cite[Thm.\@ 4.4]{KaledinLehnSorger2006SingularSymplecticModuli}---the union of all the other strata consists of strictly polystable sheaves and is the singular locus of \(M_{\vv}(S,H)\)~\cite[Prop.\@ 6.1]{KaledinLehnSorger2006SingularSymplecticModuli}. Note that when \(M_{\vv}(S,H)\) is a symplectic variety in the sense of Definition~\ref{def:zoo-symplectic-varieties}---this is known, for instance, when \(H\) is \(\vv\)-generic by Theorem~\ref{thm:global-geometry-Mv-delPezzo}---then the above stratification by polystable type is the stratification by symplectic leaves of~\cite[Thm.\@ 2.3]{Kaledin2006PoissonSymplecticSingularities}.

\section{Beauville--Mukai systems on K3-del Pezzo double covers}
\label{sec:BM-delPezzo-double-covers}

This section studies Beauville--Mukai systems on K3-del Pezzo double covers, realised as relative compactified Jacobians of the pull-back of the anticanonical linear system. Our chosen polarisation is typically non-generic, and the resulting moduli spaces are singular. We analyse their global geometry and describe their singular locus via the stratification by polystable type and the chamber-to-wall partial resolutions from adjacent stability chambers.

\subsection{2-elementary K3 surfaces and right DPN pairs}
We recall the standard dictionary between 2-elementary K3 surfaces and right DPN pairs. Following~\cite[Def.\@ 1.4]{Yoshikawa2004K3InvolutionAnalyticTorsion}, a \emph{2-elementary K3 surface} is a pair \((S,\iota)\), where \(S\) is a K3 surface and \(\iota\) is an anti-symplectic involution. We denote the invariant lattice by
\[
  L_+\coloneqq H^2(S,\mathbb{Z})^\iota \subset H^2(S,\mathbb{Z}).
\]
Since \(\iota\) acts as \(-1\) on \(H^{2,0}(S)\), one has \(L_+\subset \NS(S)\). Recall that an even lattice \(L\) is called \emph{2-elementary} if its discriminant group \(A_L\coloneqq L^\vee/L\) is a 2-elementary abelian group, namely \(A_L\cong(\mathbb{Z}/2\mathbb{Z})^a\) for some \(a\geq 0\). Nikulin's lattice-theoretic description~\cite{Nikulin1983FactorGroupsAutomorphismsHyperbolicForms} says that \(L_+\) is a primitive, hyperbolic, even, and 2-elementary sublattice of the K3 lattice \(\Lambda_{\mathrm{K3}}\). Conversely, by Torelli one can realise any primitive hyperbolic 2-elementary sublattice \(L\subset \Lambda_{\mathrm{K3}}\) as \(\NS(S)\) for a K3 surface \(S\), together with an anti-symplectic involution \(\iota\) whose cohomological action is the identity on \(L\) and \(-\mathrm{id}\) on \(L^\perp\); in particular \(L=H^2(S,\mathbb{Z})^\iota\)~\cite[\textsection4.2]{Nikulin1983FactorGroupsAutomorphismsHyperbolicForms}. 

\smallskip
The isometry class of \(L_+\), and hence the deformation type of \((S,\iota)\), is encoded by Nikulin's \emph{main invariant}
\[
  (r,a,\delta)\in \mathbb{N}\times \mathbb{N}\times \mathbb{Z}/2\mathbb{Z},
\]
where \(r=\rk L_+\), \(a\) is the length of \(A_{L_+}\), and \(\delta\) records the parity of the discriminant form~\cite[Thm.\@ 4.3.2]{Nikulin1983FactorGroupsAutomorphismsHyperbolicForms}. The admissible triples are listed in~\cite[Table~1]{Nikulin1983FactorGroupsAutomorphismsHyperbolicForms}; there are \(75\) of them, hence \(75\) deformation types of 2-elementary K3 surfaces. The corresponding moduli spaces \(\mathcal{M}_{(r,a,\delta)}\) are irreducible of dimension \(20-r\)~\cite[Rmk.\@ 4.5.3]{Nikulin1983FactorGroupsAutomorphismsHyperbolicForms},~\cite[\textsection1.4]{Yoshikawa2004K3InvolutionAnalyticTorsion}, and are closely related to Dolgachev's moduli spaces of ample lattice-polarised K3 surfaces~\cite{Dolgachev1996MirrorSymmetryLatticePolarizedK3}. For a primitive 2-elementary hyperbolic sublattice \(L\subset \Lambda_{\mathrm{K3}}\) with main invariant \((r,a,\delta)\), both theories use (an open subset of) the type-IV period domain
\[
  \Omega_L\coloneqq \{[\sigma]\in \mathbb{P}(L_\mathbb{C}^\perp)\mid (\sigma,\sigma)=0,\;(\sigma,\overline{\sigma})>0\},
\]
which is the Noether--Lefschetz locus in the full period domain where \(L\) remains algebraic. The marked K3 surface \(S\) associated with its very general point has \(\NS(S)\cong L\); compare~\cite[Cor.\@ 2.3]{Dolgachev1996MirrorSymmetryLatticePolarizedK3} and~\cite[\textsection1.4(a)]{Yoshikawa2004K3InvolutionAnalyticTorsion}.

\smallskip
On the geometric side, the fixed locus \(S^\iota\) is a disjoint union of smooth irreducible curves, and its geometry is again determined by the main invariant~\cite[\textsection4.2]{Nikulin1983FactorGroupsAutomorphismsHyperbolicForms}. If \(S^\iota=\varnothing\), equivalently \((r,a,\delta)=(10,10,0)\), then the quotient \(T\coloneqq S/\langle\iota\rangle\) is an Enriques surface and \(L_+\) is the image of \(H^2(T,\mathbb{Z})_{\mathrm{free}}\) under pull-back. When \(S^\iota\) is non-empty, \(T\) is a smooth rational surface and the quotient morphism \(f\colon S\to T\) is a double cover branched along the smooth divisor \(B\coloneqq f(S^\iota)\in |-2K_T|\). Conversely, a smooth rational surface \(T\) together with a smooth divisor \(B\in |-2K_T|\), called a \emph{right DPN pair} in~\cite[Def.\@ 2.1]{AlexeevNikulin2006DelPezzoK3}, determines a K3 double cover with anti-symplectic covering involution.

\smallskip
We consider the cases in which the quotient \(T\coloneqq S/\langle\iota\rangle\) is a smooth del Pezzo surface, i.e.\@ \(-K_T\) is ample. In terms of Nikulin's description of the fixed locus, this implies that \(R\coloneqq S^\iota\) is a single smooth curve of genus \(d+1\), where \(d\coloneqq K_T^2=10-r\) is the degree of \(T\). The corresponding cases are:
\[
\begin{array}{c|c|c|c|c}
  (r,a,\delta) & T & d\coloneqq K_T^2 & \rho(T)=r & \dim \mathcal{M}_{(r,a,\delta)} \\
  \hline
  (1,1,1) & \mathbb{P}^2 & 9 & 1 & 19 \\
  (2,2,0) & \mathbb{P}^1\times\mathbb{P}^1 & 8 & 2 & 18 \\
  (r,r,1),\ 2\leq r\leq 9 & \mathrm{Bl}_{r-1}\mathbb{P}^2 & 10-r & r & 20-r
\end{array}
\]
In the last row \(\mathrm{Bl}_{r-1}\mathbb{P}^2\) denotes the blow-up of the projective plane at \(r-1\) points in \emph{del Pezzo general position}~\cite[\textsection8.1.3]{Dolgachev2012ClassicalAlgebraicGeometry}. We call the resulting quotient morphism \(f\colon S\to T\), or equivalently the right DPN pair \((T,B)\), a \emph{K3-del Pezzo double cover}. For a very general such double cover---namely a very general point of \(\mathcal{M}_{(r,a,\delta)}\)---one has
\[
  \NS(S)=f^\ast\NS(T).
\]
The reason we restrict to these cases is that the anticanonical system \(|-K_T|\) is a moving family whose general member is a smooth genus-one curve. Hence, over a general \(E\in |-K_T|\), the curve \(C=f^{-1}(E)\) is a genus-\(d+1\) double cover of a genus-one curve, giving the uniform Prym geometry used below.

\subsection{Beauville--Mukai systems on K3-del Pezzo double covers}
Let \(f\colon S\to T\) be a very general K3-del Pezzo double cover, where \(T\) is a degree \(d\) del Pezzo surface. The morphism \(f\) is branched along a smooth curve \(B\in |-2K_T|\). As above, denote by \(\iota\in \mathrm{Aut}(S)\) the covering involution, and let \(H\coloneqq -f^\ast K_T\in \NS(S)\). Then \(H\) is an \(\iota\)-invariant polarisation of degree \(H^2=2d\). We set
\[ \vv\coloneqq (0, H, -d)\in \widetilde H_{\mathrm{alg}}(S,\mathbb{Z}) \]
and consider the moduli space \(M_{\vv}(S, H)\) of \(H\)-semistable sheaves on \(S\) with Mukai vector \(\vv\). It is a symplectic variety of dimension \(2d+2\) by~\cite[Prop.\@ 3.5]{BrakkeeEtAl2026RelativePrym}, whether or not \(H\) lies on a \(\vv\)-wall; reducedness and normality follow from the local description in \textsection\ref{subsec:Kuranishi-local-model} and the combinatorics of the associated Ext-quivers, using~\cite{CrawleyBoevey2001MomentMapQuivers, CrawleyBoevey2003MarsdenWeinsteinNormality}. It carries a natural Lagrangian fibration
\begin{equation}
  \label{eq:Beauville-Mukai-system}
  \pi\colon M_{\vv}(S,H)\to |H|\cong \mathbb{P}^{d+1}
\end{equation}
given by sending a sheaf to its Fitting support. This is the \emph{Beauville--Mukai system} associated with \((S,H)\). Over the open locus of smooth curves in \(|H|\), the fibration is the relative Jacobian of the family of genus \(d+1\) curves, and \(M_{\vv}(S,H)\) provides a symplectic compactification by rank-one torsion-free sheaves on the singular members. Since line bundles supported on integral curves are stable for any polarisation, the stable locus \(M^s_{\vv}(S,H)\) is non-empty.

\smallskip
When \(d=1\), every member of \(|H|\) is integral. Hence every pure sheaf with Mukai vector \(\vv\) is stable for any polarisation, so \(M_{\vv}(S,H)\) is independent of \(H\in \Amp(S)\) and agrees with the \emph{intrinsic} relative compactified Jacobian over \(|H|\) of~\cite{AltmanKleiman1980CompactifyingPicardSchemeI}; in particular \(M^s_{\vv}(S,H)=M_{\vv}(S,H)\) is smooth. 

When \(d\geq 2\), there are reducible members of \(|-K_T|\), and hence of \(|H|\). In this case different stability chambers give different compactifications of the relative Jacobian over the non-integral locus. Since our choice of polarisation is not \(\vv\)-generic for \(2\leq d\leq 8\), we have \(M^s_{\vv}(S,H)\subsetneq M_{\vv}(S,H)\), and the moduli space is singular. Moreover, when \(T\cong\mathbb{P}^2\) or \(\mathbb{P}^1\times \mathbb{P}^1\), the Mukai vector \(\vv=(0, H, -d)\in \widetilde H_{\mathrm{alg}}(S,\mathbb{Z})\) is not primitive, so some singularities are structural and do not depend on the choice of polarisation.

\smallskip
In the following we analyse the singularities of \(M_{\vv}(S,H)\) in terms of the stratification by polystable type. We first describe the non-integral supports, reducing the problem to splittings of \(-K_T\in \NS(T)\); then we translate these splittings into minimal polystable strata; finally we use the induced walls and chamber-to-wall morphisms to compute the global invariants. We begin with a simple fact about linear systems of arithmetic genus zero on del Pezzo surfaces.

\begin{lemma}
  \label{lem:delPezzo-linear-systems}
  Let \(D\in \NS(T)\) be a non-zero effective divisor class with \(p_a(D)=0\) on a smooth del Pezzo surface \(T\). If the general member of \(|D|\) is connected, then it is a smooth rational curve and we have
  \[ h^0(T, \mathscr{O}_T(D))=D^2+2. \]
\end{lemma}

\begin{proof}
  By adjunction we have \( -K_T\cdot D=D^2+2\). Since \(-K_T\) is ample and \(D\neq 0\) is effective, we have \(D^2\geq -1\).

  If \(D^2=-1\), then \(-K_T\cdot D=1\). Hence any member of \(|D|\) is irreducible. It is therefore a smooth rational \((-1)\)-curve. Moreover \(|D|\) has only one member, since two distinct members would have non-negative intersection, contradicting \(D^2=-1\). Hence \(h^0(T,\mathscr{O}_T(D))=1=D^2+2\).

  Assume now that \(D^2\geq 0\). Riemann--Roch and \(H^0(T,\mathscr{O}_T(K_T-D))=0\) give
  \[
    h^0(T,\mathscr{O}_T(D))\geq \chi(T,\mathscr{O}_T(D))=D^2+2\geq 2.
  \]
  Thus \(|D|\) has a non-trivial moving part. Write \(|D|=|M|+F\), where \(F\) is the fixed part. The moving part \(M\) has no fixed components, hence is nef. Therefore Kodaira vanishing gives \(H^1(T,\mathscr{O}_T(M))=0\). We claim that \(F=0\). Suppose otherwise, and let \(E\) be an irreducible component of \(F\). From the isomorphism \( H^0(T,\mathscr{O}_T(M))\cong H^0(T,\mathscr{O}_T(M+E)) \) and the exact sequence
  \[
    0\to \mathscr{O}_T(M)\to \mathscr{O}_T(M+E)\to \mathscr{O}_E(M+E)\to 0
  \]
  we deduce that \( H^0(E,\mathscr{O}_E(M+E))=0 \). This forces \(E^2<0\). Indeed, if \(E^2\geq 0\), then
  \[
    \deg\mathscr{O}_E(M+E)-\deg\omega_E=-K_T\cdot E+M\cdot E>0,
  \]
  so \(H^1(E,\mathscr{O}_E(M+E))=0\); Riemann--Roch on the integral Gorenstein curve \(E\) would then give
  \[
    h^0(E,\mathscr{O}_E(M+E))
    =
    M\cdot E+\frac{E^2-K_T\cdot E}{2}>0,
  \]
  a contradiction. Since \(-K_T\) is ample, adjunction now gives \(E^2=-1\) and \(E\cong \mathbb{P}^1\). Therefore \(H^0(E,\mathscr{O}_E(M+E))=0\) forces \(M\cdot E=0\). This holds for every irreducible component of \(F\), so \(M\cdot F=0\). The general member of \(|D|\) is therefore the disjoint union of a member of \(|M|\) and \(F\), contradicting the connectedness assumption. Hence \(F=0\), and \(D=M\) is nef. By~\cite[Cor.\@ 3.4]{Harbourne1985LinearSystemsRationalSurfaces}, the nef divisor \(D\) is base-point-free. Bertini's theorem then shows that the general member of \(|D|\) is smooth. It is connected by assumption and has arithmetic genus zero, hence it is a smooth rational curve. Finally, Kodaira vanishing applied to \(D-K_T\) gives \(H^1(T,\mathscr{O}_T(D))=0\), while \(H^0(T,\mathscr{O}_T(K_T-D))=0\). Riemann--Roch gives
  \[
    h^0(T,\mathscr{O}_T(D))=\chi(T,\mathscr{O}_T(D))=D^2+2. \qedhere
  \]
\end{proof}

%Easy to see this fails on K3s for instance, by choosing two smooth rational curves C_1, C_2 meeting at a point and then D=C_1+C_2, then h^0(D)=1 and only consists of C_1+C_2.

\begin{proposition}\label{prop:non-integral-members-delPezzo}
  Assume \(d\geq 2\). The irreducible components of the non-integral locus of \(|H|\) are the images of the addition morphisms
  \[ |f^\ast D_1|\times |f^\ast D_2|\to |H|, \quad (C_1, C_2)\mapsto C_1+C_2, \]
  where \(-K_T=D_1+D_2\) is a splitting of the anticanonical class \(-K_T\in \NS(T)\) into a sum of effective classes \(D_i\in \NS(T)\) whose general members are smooth rational curves with \(D_1\cdot D_2=2\). Every such component has codimension \(3\) in \(|H|\).
\end{proposition}

\begin{proof}
  Let \(C\coloneqq \sum_{i=1}^k C_i \in |H|\) be a reducible curve, where each \(C_i\) is integral, so \(k\geq 2\). Since \(\NS(S)=f^\ast \NS(T)\), we have \(C_i\in |f^\ast D_i|\) for some effective classes \(D_i\in \NS(T)\) with \(-K_T=\sum_{i=1}^k D_i\). Moreover, for any such effective summand \(D_i\) of \(-K_T\), pull-back of sections induces an isomorphism
  \begin{equation}\label{eq:pull-back-sections-delPezzo}
    f^\ast \colon H^0(T, \mathscr{O}_T(D_i))\iso H^0(S, \mathscr{O}_S(f^\ast D_i)),
  \end{equation}
  since \(H^0(T, \mathscr{O}_T(D_i+K_T))=0\) as \(D_i+K_T\) is anti-effective. In particular, the general member of each \(|D_i|\) is integral.

  If \(k=2\), then \(C\) belongs to the image of the addition morphism \(|f^\ast D_1|\times |f^\ast D_2|\to |H|\). By adjunction,
  \[
    p_a(D_1)=p_a(D_2)=1-\frac{1}{2}D_1\cdot D_2.
  \]
  Since the general member of \(|D_i|\) is connected, \(p_a(D_i)\geq 0\). Also \(D_1\cdot D_2 \geq 1\), since \(-K_T\) is ample; as \(D_1\cdot D_2\) is even, \(D_1\cdot D_2\geq 2\). Together these inequalities give \(D_1\cdot D_2=2\) and \(p_a(D_i)=0\) for \(i=1,2\), so the claim follows from Lemma~\ref{lem:delPezzo-linear-systems}.

  If \(k\geq 3\), since \(-K_T\) is ample, after relabelling we may assume that \(E\coloneqq \sum_{i\geq 2} D_i\) has connected general member and that \(D_1\cdot E\geq 1\). The same argument gives \(D_1\cdot E=2\) and \(p_a(E)=p_a(D_1)=0\), and the general member of \(|E|\) is integral by Lemma~\ref{lem:delPezzo-linear-systems}. Hence \(C\) belongs to the image of the addition morphism \(|f^\ast D_1|\times |f^\ast E|\to |H|\).

  Finally, since \(|f^\ast D_1|\times |f^\ast D_2|\to |H|\) is finite, we can compute the codimension of its image using~\eqref{eq:pull-back-sections-delPezzo} and Lemma~\ref{lem:delPezzo-linear-systems}:
  \begin{align*}
    \begin{split}
      \dim |H|- (\dim |f^\ast D_1| + \dim |f^\ast D_2|) & = (K_T^2+1)-(D_1^2+D_2^2+2)    \\
                                                        & = 2 D_1\cdot D_2-1=3. \qedhere
    \end{split}
  \end{align*}
\end{proof}

We now describe the minimal strictly polystable strata of \(M_{\vv}(S,H)\), with respect to the stratification by polystable type from Lemma~\ref{lem:stratification-polystable-type}.

\begin{proposition}\label{prop:polystable-strata-delPezzo}
  Assume \(d\geq 2\). The minimal non-empty strata in \(M_{\vv}(S,H)\setminus M^s_{\vv}(S,H)\) are those of type
  \begin{equation}\label{eq:minimal-polystable-type-delPezzo}
    \mathfrak{t}\coloneqq (1, \vv_1; 1, \vv_2)
  \end{equation}
  where \(\vv_i\coloneqq (0, f^\ast D_i, D_i\cdot K_T)\in \widetilde H_{\mathrm{alg}}(S,\mathbb{Z})\), with \(D_i\in \NS(T)\) as in Proposition~\ref{prop:non-integral-members-delPezzo}. For any such type \(\mathfrak{t}\), one has
  \begin{equation*}
    \mathrm{codim}(M_{\vv}^\mathfrak{t}(S,H), M_{\vv}(S,H))=6.
  \end{equation*}
\end{proposition}

\begin{proof}
  Any minimal type strictly above the stable type \((1, \vv)\) must be of the form~\eqref{eq:minimal-polystable-type-delPezzo}, with \(\vv_1+\vv_2=\vv\). Since \(\NS(S)=f^\ast \NS(T)\), we can write \(\vv_i\coloneqq (0, f^\ast D_i, \chi_i)\) for suitable effective classes \(D_i\in \NS(T)\) with \(D_1+D_2=-K_T\) and \(\chi_i \in \mathbb{Z}\). The equality of reduced Hilbert polynomials forces \(\chi_i=D_i\cdot K_T\) for \(i=1,2\). The result then follows from Proposition~\ref{prop:non-integral-members-delPezzo} and the fact that each \(M^s_{\vv_i}(S,H)\) is non-empty since the general member of $|f^\ast D_i|$ is integral.

  Finally, \(M_{\vv}^\mathfrak{t}(S,H)\) is an open subset of the image of the finite morphism~\eqref{eq:direct-sum-morphism}, and each \(M_{\vv_i}(S,H)\) admits a Lagrangian fibration onto \(|f^\ast D_i|\). The codimension statement therefore follows from Proposition~\ref{prop:non-integral-members-delPezzo}.
\end{proof}

\begin{remark}[Explicit admissible splittings of \(-K_T\) and associated walls] \label{rmk:explicit-minimal-polystable-strata-delPezzo}
  The minimal polystable strata in \(M_{\vv}(S,H)\setminus M^s_{\vv}(S,H)\) can be described explicitly by listing the minimal \emph{admissible} splittings of \(-K_T\in \NS(T)\), namely those inducing non-empty strata in \(M_\vv(S,H)\) as in Proposition~\ref{prop:non-integral-members-delPezzo}, and the walls in \(\Amp(S)\) they define.

  For an admissible splitting \(-K_T=D_1+D_2\in \NS(T)\), set
  \begin{equation}\label{eq:wall-class} 
    \delta_{D_1, D_2}\coloneqq (D_1\cdot K_T) f^\ast D_2 - (D_2\cdot K_T) f^\ast D_1\in \NS(S).
  \end{equation}
  When this class is non-zero, the equation \(\delta_{D_1,D_2}\cdot\alpha=0\) defines a \(\vv\)-wall in \(\Amp(S)\)~\cite[\textsection2.1]{PeregoRapagnetta2023ISVModuliSheaves}. This happens precisely when \(D_1\) and \(D_2\) are not proportional after normalising by their anticanonical degrees, or equivalently when the corresponding Mukai vectors are not proportional---for instance, whenever \(\vv\in \widetilde H_{\mathrm{alg}}(S,\mathbb{Z})\) is primitive. 

  \smallskip
  For \(T\cong\mathbb{P}^2\), the only admissible minimal splitting is \(-K_T=(2h)+(h)\), where \(h\in \NS(T)\) is the class of a line. Hence there is a unique irreducible component of the singular locus of codimension 6. This splitting has \(\delta_{2h,h}=0\), so it does not define a wall; the singularities of \(M_\vv(S,H)\) are structural and due to the non-primitivity of \(\vv\). See also Figure~\ref{fig:singular-locus-P2} for the full stratification.

  % Hasse diagram of the stratification by polystable type 
\begin{figure}
  \centering
  \begin{tikzcd}[row sep=1.8em]
    {(1, 3\vv_0)} \\
    {(1, 2\vv_0; 1, \vv_0)} \\
    {(1, \vv_0; 1, \vv_0; 1,\vv_0)} \\
    {(2, \vv_0; 1, \vv_0)} \\
    {(3,\vv_0)}
    \arrow["{\scriptscriptstyle 6}", from=2-1, to=1-1]
    \arrow["{\scriptscriptstyle 2}", from=3-1, to=2-1]
    \arrow["{\scriptscriptstyle 4}", from=4-1, to=3-1]
    \arrow["{\scriptscriptstyle 4}", from=5-1, to=4-1]
  \end{tikzcd}
  \caption{\emph{Polystable-type Hasse diagram for \(M_{\vv}(S,H)\), with \(T\cong\mathbb{P}^2\)}. Here \(\vv=3\vv_0\), with \(\vv_0\coloneqq (0, f^\ast h, -3)\). Nodes are polystable types; an upward arrow labelled \(k\) means that the lower stratum is contained in the closure of the upper one with codimension \(k\). Note that \(M_{\vv_0}(S,H)\) is a smooth hyper-K\"ahler fourfold of \(\mathrm{K3}^{[2]}\)-type.}
  \label{fig:singular-locus-P2}
\end{figure}
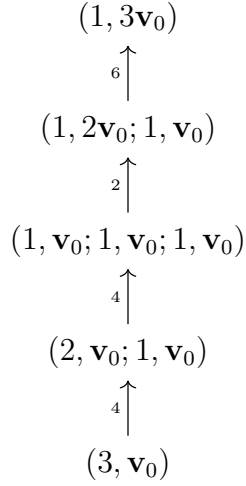
  
  \smallskip
  For \(T\cong\mathbb{P}^1\times \mathbb{P}^1\), writing \(\NS(T)=\mathbb{Z}f_1\oplus \mathbb{Z}f_2\), where \(f_1,f_2\in \NS(T)\) are the classes of the fibres, the only admissible minimal splittings of \(-K_T\) are given by
  \begin{equation*}
    -K_T=(f_1+f_2)+(f_1+f_2), \quad (f_1)+(f_1+2f_2), \quad (2f_1+f_2)+(f_2).
  \end{equation*}
  The case \(D_1\coloneqq 2f_1\) and \(D_2 \coloneqq 2f_2\) is not admissible, since the general members of \(|2f_1|\) and \(|2f_2|\) are not connected. Hence there are three irreducible components of the singular locus of codimension 6 in this case. The first splitting has zero wall class and comes from the non-primitivity of \(\vv\), while the last two splittings define the same non-zero wall through \(H\). See also Remark~\ref{rmk:degree-8-comparison} and Figure~\ref{fig:singular-locus-P1xP1} for the full stratification.
  
  \smallskip
  Finally, let \(T\cong\mathrm{Bl}_{r-1}(\mathbb{P}^2)\) for \(2\leq r\leq 9\), and write
  \begin{equation*}
    \NS(T)=\mathbb{Z}h\oplus \bigoplus_{i=1}^{r-1} \mathbb{Z}e_i,
  \end{equation*}
  where \(h\in \NS(T)\) is the pull-back of the hyperplane class of \(\mathbb{P}^2\) and the \(e_i\in \NS(T)\) are the exceptional classes.
  A simple coefficient count, together with the usual general-position conditions for del Pezzo blow-ups~\cite[\textsection8.1.3]{Dolgachev2012ClassicalAlgebraicGeometry}, shows that the admissible minimal splittings of the anticanonical class
  \[ -K_T=3h- \sum_{i=1}^{r-1} e_i \in \NS(T)\]
  fall into the following two types:
  \begin{itemize}
    \item Type I: \(D_1\coloneqq e_i\) and \(D_2\coloneqq 3h-2e_i-\sum_{j\neq i} e_j\), for some \(i\in \{1, \dots, r-1\}\);
    \item Type II: \(D_1\coloneqq h-\sum_{i\in I} e_i\) and \(D_2\coloneqq 2h-\sum_{j\notin I} e_j\), where \(I\subset \{1, \dots, r-1\}\) has cardinality \(\max\{0,r-6\}\leq |I|\leq 2\).
  \end{itemize}
  Here \(|I|\leq 2\) and \(r-1-|I|\leq 5\) are precisely the conditions that the corresponding line and conic systems have connected general members.
  For \(r=9\), the type I residual class \(D_2\) satisfies \(-K_T\cdot D_2=0\); since \(-K_T\) is ample, it is not effective. Thus no type I splitting occurs in degree \(1\), which is consistent with the smoothness of \(M_\vv(S,H)\). Hence the number of strata of type I is \(r-1\) for \(2\leq r\leq 8\) and \(0\) for \(r=9\), while the number of strata of type II is
  \[
    \sum_{m=\max\{0,r-6\}}^2 \binom{r-1}{m}.
  \]
  The total number of irreducible components of the singular locus of \(M_\vv(S,H)\) is the sum of the two contributions, as summarised in the following table:
  \begin{center}
    \begin{tabular}{c|c|c|c|c}
      \(r\) & \(d=K_T^2\) & Type I & Type II & Total \\
      \hline
      2     & 8           & 1      & 2       & 3     \\
      3     & 7           & 2      & 4       & 6     \\
      4     & 6           & 3      & 7       & 10    \\
      5     & 5           & 4      & 11      & 15    \\
      6     & 4           & 5      & 16      & 21    \\
      7     & 3           & 6      & 21      & 27    \\
      8     & 2           & 7      & 21      & 28    \\
      9     & 1           & 0      & 0       & 0
    \end{tabular}
  \end{center}
  In the case \(r=8\), so \(d=2\), this recovers the 28 isolated singular points which appear in~\cite[Prop.\@ 1.2]{MarkushevichTikhomirov2007Prymian}, while for \(r=7\), so \(d=3\), this is consistent with the 27 K3 surfaces described by~\cite[Lemma 3.1]{Matteini2016SingularPrymFibration}.

  \smallskip
  In the blow-up cases with \(2\leq r\leq 8\), all these splittings give non-trivial walls through \(H\).
  For a splitting of type I, this gives \(\delta_i\coloneqq d f^\ast e_i-H\), up to sign. A direct computation shows that these wall classes are linearly independent. Since they are orthogonal to \(H\), they form a basis of the real hyperplane
  \[
    H^\perp_{\mathbb{R}}\coloneqq \{\alpha\in \NS(S)_{\mathbb{R}}\mid \alpha\cdot H=0\}\subset \NS(S)_{\mathbb{R}}.
  \]
  In particular, the wall classes coming from splittings of type II belong to the same subspace \(H^\perp_{\mathbb{R}}\), and hence are linear combinations of the type I classes. The common intersection of all such walls is the ray \(\mathbb{R}_{\geq 0}H\), which is then a codimension-\((r-1)\) face of the local \(\vv\)-wall configuration in \(\Amp(S)\).
\end{remark}

The next proposition records the global properties of the Beauville--Mukai systems \(M_{\vv}(S,H)\) used below: their position among the classes of symplectic varieties from Definition~\ref{def:zoo-symplectic-varieties}, their numerical invariants \(b_2\) and \(\rho\), and the structure of their singular locus.

\begin{proposition}
  \label{prop:geometry-Mv-delPezzo}
  \begin{enumerate}
    \item For \(d=1\) the moduli space \(M_{\vv}(S,H)\) is a smooth hyper-K\"ahler fourfold of \(\mathrm{K3}^{[2]}\)-type with \(b_2(M_{\vv}(S,H))=23\) and \(\rho(M_{\vv}(S,H))=10\).
    \item For \(2\leq d\leq 8\), the moduli space \(M_{\vv}(S,H)\) is an irreducible symplectic variety of dimension \(2d+2\). It is not \(\mathbb{Q}\)-factorial, and it has simply connected regular locus and
    \[ 
      b_2(M_{\vv}(S,H))=14+d, \quad \rho(M_{\vv}(S,H))=2.
    \]
    Its singular locus is pure of dimension \(2d-4\), and its irreducible components are parametrised by the minimal splittings of the Mukai vector \(\vv \in \widetilde H_{\mathrm{alg}}(S,\mathbb{Z})\) from Proposition~\ref{prop:polystable-strata-delPezzo}.
    \item For \(d=9\), the moduli space \(M_{\vv}(S,H)\) is a locally factorial irreducible symplectic variety of dimension \(20\) with simply connected regular locus and
    \[ 
      b_2(M_{\vv}(S,H))=23, \quad \rho(M_{\vv}(S,H))=2.
    \]
    Its singular locus is irreducible of dimension \(14\). 
  \end{enumerate}
\end{proposition}

\begin{proof}
  For \(d=1\), this follows from the fact that every member of \(|H|\) is integral, so that \(M_{\vv}(S,H)=M^s_{\vv}(S,H)\) is smooth, and from Theorems~\ref{thm:global-geometry-Mv-delPezzo} and~\ref{thm:2nd-cohomology-moduli-sheaves}. In fact, since \(\chi=-1\), there are no \(\vv\)-walls in \(\Amp(S)\).  

  \smallskip
  Assume first that \(d=8\) and \(T\cong \mathbb{P}^1\times \mathbb{P}^1\), so that \(\NS(T)=\mathbb{Z}f_1\oplus \mathbb{Z}f_2\), where \(f_1\) and \(f_2\) are the classes of the fibres. The splittings
  \[ -K_T= (f_1)+(f_1+2f_2), \quad (2f_1+f_2)+(f_2) \]
  both induce the same flopping wall \(\mathbb{R}_{\geq 0}H\) in \(\Amp(S)\). For a polarisation \(H^+\) off the wall, \cite[Prop.\@ 2.5]{ArbarelloSacca2018K3QuiverSingularities} gives a crepant birational chamber-to-wall morphism
  \[ \phi\colon M_\vv(S, H^+)\to M_{\vv}(S,H). \]
  The variety \(M_\vv(S,H)\) is hence cohomologically irreducible symplectic by Proposition~\ref{prop:criteria-ISV}. Moreover, by semismallness~\cite[Prop.\@ 2.16]{Tighe2025LooijengaLuntsVerbitskyAlgebra}, $\phi$ is a small terminalisation, as it is an isomorphism over an open subset whose complement has codimension 6. It follows that \(M_\vv(S,H)\) is not \(\mathbb{Q}\)-factorial~\cite[Cor.\@ 2.63]{KollarMori1998BirationalGeometry}.  Moreover, Theorem~\ref{thm:global-geometry-Mv-delPezzo} identifies the regular locus of \(M_\vv(S,H^+)\) with the stable locus \(M^s_\vv(S,H^+)\) and shows that it is simply connected. Since the open immersion 
  \[ \phi^{-1}(M^s_\vv(S, H))\subset M^s_\vv(S, H^+) \]
  has complement of codimension \(\geq 3\), the regular locus \(M^{\mathrm{reg}}_\vv(S,H)=M^s_\vv(S,H)\) is also simply connected. Hence \(M_{\vv}(S,H)\) is an irreducible symplectic variety by Proposition~\ref{prop:criteria-ISV}. 
  We have \(\rho(M_\vv(S,H))\geq 2\), since \(\Pic(M_\vv(S,H))\) contains an ample class and the class of the Lagrangian fibration \(\pi\colon M_{\vv}(S,H)\to |H|\). On the other hand, Theorem~\ref{thm:2nd-cohomology-moduli-sheaves} gives \(\rho(M_\vv(S,H^+))=3\), so the small contraction \(\phi\) has relative Picard number
  \[  \rho(M_\vv(S,H^+)/M_\vv(S,H))=1.\]
  The morphism \(\phi^\ast\colon H^2(M_\vv(S, H), \mathbb{Z})\to H^2(M_\vv(S, H^+), \mathbb{Z})\) is injective and induces an isomorphism on transcendental lattices~\cite[Lemma 2.1]{BakkerLehn2021GlobalTorelliSingularSymplectic}. Its cokernel is therefore algebraic~\cite[Lemma 5.21]{BakkerLehn2022GlobalModuli}, and after tensoring with \(\mathbb{Q}\) it has dimension
  \[
    \rho(M_\vv(S,H^+))-\rho(M_\vv(S,H))=\rho(M_\vv(S,H^+)/M_\vv(S,H))=1.
  \]
  Therefore Theorem~\ref{thm:2nd-cohomology-moduli-sheaves} gives
  \[
    b_2(M_{\vv}(S,H))=b_2(M_{\vv}(S,H^+))-1=22.
  \]
  The last statement follows from Proposition~\ref{prop:polystable-strata-delPezzo} and Remark~\ref{rmk:explicit-minimal-polystable-strata-delPezzo}.

  \smallskip
  It remains to treat the case \(T\cong\mathrm{Bl}_{r-1} \mathbb{P}^2\), where \(2\leq r\leq 8\) and \(d=10-r\). Here \(H\) lies on a codimension-\((r-1)\) face of the \(\vv\)-wall configuration in \(\Amp(S)\), but \(\vv\) is primitive. As above, let \(H^+\) be a \(\vv\)-generic polarisation in an adjacent chamber with small crepant resolution \(\phi\colon M_\vv(S, H^+)\to M_\vv(S,H)\). Since Theorem~\ref{thm:2nd-cohomology-moduli-sheaves} gives \(\rho(M_\vv(S, H^+))=r+1\), by the same argument as above it remains to show that 
  \begin{equation}\label{eq:relative-Picard-number-blowups}
    \rho(M_\vv(S,H^+)/M_\vv(S,H))\geq r-1. 
  \end{equation}
  To that end, we exhibit \(r-1\) linearly independent curve classes in the relative space \(N_1(M_\vv(S,H^+)/M_\vv(S,H))\). For each component of type I of the singular locus of \(M_\vv(S,H)\) described in Remark~\ref{rmk:explicit-minimal-polystable-strata-delPezzo}, pick a general point \(\mathscr{F}_i\oplus \mathscr{G}_i\) in the corresponding stratum, with 
  \[ \vv(\mathscr{F}_i)=(0, f^\ast e_i, -1), \quad \vv(\mathscr{G}_i)=(0, f^\ast (-K_T-e_i), 1-d).  \]
  We may choose \(H^+\) so that, for each \(i\), the fibre of \(\phi\) over \(\mathscr{F}_i\oplus \mathscr{G}_i\) is the \(3\)-dimensional projective space of extensions \(\mathbb{P}\Ext^1(\mathscr{F}_i, \mathscr{G}_i)\subset M_\vv(S, H^+)\). Choose a line \(\ell_i\subset \mathbb{P}\Ext^1(\mathscr{F}_i, \mathscr{G}_i)\) for each \(i=1, \dots, r-1\). We claim that the classes of the curves \(\ell_i\) are linearly independent in \(N_1(M_\vv(S,H^+)/M_\vv(S,H))\). To see this, let \(\delta_i\coloneqq d f^\ast e_i-H\in \NS(S)\) be the classes~\eqref{eq:wall-class} defining the walls of type I. The classes
  \[ \mathbf{w}_i\coloneqq (0, \delta_i, 0)\in \vv^\perp, \]
  which are linearly independent over \(\mathbb{Q}\), define \(r-1\) line bundles \(\lambda_\vv(\mathbf{w}_i)\in \NS(M_\vv(S, H^+))\) via the Mukai morphism of Theorem~\ref{thm:2nd-cohomology-moduli-sheaves}. Using Example~\ref{ex:Mukai-Lepotier-extensions}, we compute the pairings
  \[ \lambda_\vv(\mathbf{w}_i)\cdot \ell_j=-\mathbf{w}_i\cdot  \vv(\mathscr{F}_j)=-\delta_i\cdot  f^\ast e_j= 2+2d \delta_{ij}, \]
  and the intersection matrix is non-degenerate. This proves the claim. As before,
  \[
    \rho(M_\vv(S, H))\geq 2,
  \]
  so \eqref{eq:relative-Picard-number-blowups} is an equality. Hence
  \[
    \rho(M_\vv(S,H))=2,\quad b_2(M_\vv(S,H))=b_2(M_\vv(S, H^+))-(r-1)=14+d.
  \]

  \smallskip
  Finally, assume \(d=9\), so \(T\cong \mathbb{P}^2\) and \(\NS(S)=\mathbb{Z}f^\ast h\), where \(h\in \NS(T)\) is the class of a line. We have \(H=3f^\ast h\) and \(\vv=3\vv_0\), where \(\vv_0\coloneqq (0, f^\ast h, -3)\in \widetilde H_{\mathrm{alg}}(S,\mathbb{Z})\) is primitive. Since \(H\) is \(\vv\)-generic, Theorem~\ref{thm:global-geometry-Mv-delPezzo} shows that \(M_{\vv}(S,H)\) is a locally factorial irreducible symplectic variety with simply connected regular locus and terminal singularities. By Theorem~\ref{thm:2nd-cohomology-moduli-sheaves} we have \(b_2(M_{\vv}(S,H))=23\) and \(\rho(M_{\vv}(S,H))=2\). The singular locus is described in Proposition~\ref{prop:polystable-strata-delPezzo} and Remark~\ref{rmk:explicit-minimal-polystable-strata-delPezzo}.
\end{proof}

The proof of Proposition~\ref{prop:geometry-Mv-delPezzo} also gives an integral identification of transcendental lattices
\begin{equation}\label{eq:transcendental-lattice-Mv-delPezzo}
  T(S)\iso T(M_\vv(S,H)).
\end{equation}
For \(d=1\) and \(d=9\), in the absence of $\vv$-walls, this is the restriction of the Mukai morphism of Theorem~\ref{thm:2nd-cohomology-moduli-sheaves}. For \(2\leq d\leq 8\), it is obtained---upon choosing a nearby $\vv$-generic polarisation $H^+$---by composing the Mukai isomorphism
\[
  \lambda_\vv^+\colon T(S)\iso T(M_\vv(S,H^+))
\]
with the inverse of the isomorphism
\[
  \phi^\ast \colon T(M_\vv(S, H))\iso T(M_\vv(S,H^+))
\]
from~\cite[Lemma 2.1]{BakkerLehn2021GlobalTorelliSingularSymplectic}.

% Hasse diagram of the stratification by polystable type Bl_1 P2
\begin{figure}
  \centering
  {\small
  \makebox[\textwidth][c]{%
  \begin{tikzcd}[ampersand replacement=\&]
	\& {(1, 2\vv_1+3\vv_2)} \& \\
	{(1, \vv_1;1, \vv_1+3\vv_2)} \& {(1, 2\vv_1+2\vv_2;1,\vv_2)} \& {(1, \vv_1+\vv_2;1,\vv_1+2\vv_2)} \\
	\& {(1, \vv_1; 1,\vv_1+2\vv_2;1,\vv_2)} \& {(1,\vv_1+\vv_2;1,\vv_1+\vv_2;1,\vv_2)} \\
	\& {(1,\vv_1;1, \vv_1+\vv_2;1,\vv_2;1,\vv_2)} \\
	{(2,\vv_1;1,\vv_2; 1,\vv_2;1,\vv_2)} \& {(1,\vv_1; 1,\vv_1+\vv_2; 2,\vv_2)} \& {(2,\vv_1+\vv_2; 1,\vv_2)} \\
	\& {(2,\vv_1; 2,\vv_2; 1,\vv_2)} \\
	\& {(2, \vv_1; 3, \vv_2)}
	\arrow["{{\scriptscriptstyle 6}}", from=2-1, to=1-2]
	\arrow["{{\scriptscriptstyle 6}}"', from=2-2, to=1-2]
	\arrow["{{\scriptscriptstyle 6}}"', from=2-3, to=1-2]
	\arrow["{{\scriptscriptstyle 2}}", from=3-2, to=2-1]
	\arrow["{{\scriptscriptstyle 2}}"', from=3-2, to=2-2]
	\arrow["{{\scriptscriptstyle 2}}"{near end}, from=3-2, to=2-3]
	\arrow["{{\scriptscriptstyle 2}}"'{near end}, from=3-3, to=2-2]
	\arrow["{{\scriptscriptstyle 2}}"', from=3-3, to=2-3]
	\arrow["{{\scriptscriptstyle 2}}", from=4-2, to=3-2]
	\arrow["{{\scriptscriptstyle 2}}"', from=4-2, to=3-3]
	\arrow["{{\scriptscriptstyle 2}}", from=5-1, to=4-2]
	\arrow["{{\scriptscriptstyle 2}}", from=5-2, to=4-2]
	\arrow["{{\scriptscriptstyle 4}}", from=5-3, to=3-3]
	\arrow["{{\scriptscriptstyle 2}}", from=6-2, to=5-1]
	\arrow["{{\scriptscriptstyle 2}}", from=6-2, to=5-2]
	\arrow["{{\scriptscriptstyle 2}}"', from=6-2, to=5-3]
	\arrow["{{\scriptscriptstyle 2}}"', from=7-2, to=6-2]
\end{tikzcd}}}
  \caption{\emph{Polystable-type Hasse diagram for \(M_{\vv}(S,H)\), with \(T\cong\mathrm{Bl}_1\mathbb{P}^2\)}. Here \(\vv=2\vv_1+3\vv_2\) as in~\eqref{eq:wall-classes-d8-blowup}. Nodes are polystable types; an upward arrow labelled \(k\) means that the lower stratum is contained in the closure of the upper one with codimension \(k\).}
  \label{fig:singular-locus-Bl1P2}
\end{figure}

\begin{remark}[Wall and structural singularities for \(d=8\)]\label{rmk:degree-8-comparison} The two degree-eight cases illustrate two different sources of singularities. First let \(T\cong\mathrm{Bl}_1\mathbb{P}^2\), so that \(\NS(T)=\mathbb{Z}h\oplus \mathbb{Z}e\). The effective cone of \(T\) is generated by \(e\) and \(h-e\), and \(-K_T=2e+3(h-e)\), as in Remark~\ref{rmk:explicit-minimal-polystable-strata-delPezzo}. Set
\begin{equation}\label{eq:wall-classes-d8-blowup}
  \vv_1\coloneqq (0,f^\ast e,-1),\qquad \vv_2\coloneqq (0,f^\ast(h-e),-2)\in \widetilde H_{\mathrm{alg}}(S,\mathbb{Z}),
\end{equation}
so that \( \vv=(0,H,-8)=2\vv_1+3\vv_2 \in \widetilde H_{\mathrm{alg}}(S,\mathbb{Z})\). Both classes \(\vv_1\) and \(\vv_2\) are admissible, in the sense that the associated moduli spaces have non-empty stable loci. Indeed, \(\vv_1^2=-2\), so \(M_{\vv_1}(S,H)\) consists of a single point, corresponding to a unique stable sheaf, while \(\vv_2^2=0\), so \(M_{\vv_2}(S,H)\) is a smooth K3 surface. 
In this case \(\vv=2\vv_1+3\vv_2\in \widetilde H_{\mathrm{alg}}(S,\mathbb{Z})\) is primitive, and the three irreducible components of the singular locus correspond to the three minimal types
\[
  (1, \vv_1; 1, \vv_1+3\vv_2), \quad (1, 2\vv_1+2\vv_2; 1, \vv_2), \quad (1, \vv_1+\vv_2; 1, \vv_1+2\vv_2).
\]
These are all genuine wall strata: the associated wall classes~\eqref{eq:wall-class} are non-zero. Thus, after moving the polarisation to a \(\vv\)-generic chamber, these components disappear, and the chamber-to-wall morphism is a small crepant resolution of \(M_\vv(S,H)\).

% Hasse diagram of the stratification by polystable type P1xP1
\begin{figure}
  \centering
  \makebox[\textwidth][c]{%
  \begin{tikzcd}[ampersand replacement=\&]
	\&\& {(1, 2\vv_1+2\vv_2)} \&\& \\
	\& {(1, 2\vv_1+\vv_2; 1, \vv_2)} \& {(1, \vv_1; 1, \vv_1+2\vv_2)} \& {(1, \vv_1+\vv_2; 1, \vv_1+\vv_2)} \\
	\&\& {(1, \vv_1+\vv_2; 1, \vv_1; 1, \vv_2)} \\
	{\phantom{X+Y+Z+W}} \&\& {(1, \vv_1; 1, \vv_1; 1, \vv_2; 1, \vv_2)} \&\& {\phantom{X+Y+Z+W}} \\
	\& {(1, \vv_1; 1, \vv_1; 2, \vv_2)} \& {(2, \vv_1; 1, \vv_2; 1, \vv_2)} \& {(2, \vv_1+\vv_2)} \\
	\&\& {(2, \vv_1; 2, \vv_2)}
	\arrow["{{\scriptscriptstyle 6}}", from=2-2, to=1-3]
	\arrow["{{\scriptscriptstyle 6}}", from=2-3, to=1-3]
	\arrow["{{\scriptscriptstyle 6}}"', from=2-4, to=1-3]
	\arrow["{{\scriptscriptstyle 2}}", from=3-3, to=2-2]
	\arrow["{{\scriptscriptstyle 2}}", from=3-3, to=2-3]
	\arrow["{{\scriptscriptstyle 2}}"', from=3-3, to=2-4]
	\arrow["{{\scriptscriptstyle 2}}", from=4-3, to=3-3]
	\arrow["{{\scriptscriptstyle 2}}", from=5-2, to=4-3]
	\arrow["{{\scriptscriptstyle 2}}"', from=5-3, to=4-3]
	\arrow["{{\scriptscriptstyle 6}}", from=5-4, to=2-4]
	\arrow["{{\scriptscriptstyle 2}}", from=6-3, to=5-2]
	\arrow["{{\scriptscriptstyle 2}}"', from=6-3, to=5-3]
	\arrow["{{\scriptscriptstyle 2}}"', from=6-3, to=5-4]
\end{tikzcd}}
  \caption{\emph{Polystable-type Hasse diagram for \(M_{\vv}(S,H)\), with \(T\cong\mathbb{P}^1\times \mathbb{P}^1\)}. Here \(\vv=2\vv_1+2\vv_2\) as in~\eqref{eq:wall-classes-d8-quadric}. Nodes are polystable types; an upward arrow labelled \(k\) means that the lower stratum is contained in the closure of the upper one with codimension \(k\).}
  \label{fig:singular-locus-P1xP1}
\end{figure}

\smallskip
For \(T\cong\mathbb{P}^1\times \mathbb{P}^1\), we have \(\NS(T)=\mathbb{Z}f_1\oplus \mathbb{Z}f_2\) as in Remark~\ref{rmk:explicit-minimal-polystable-strata-delPezzo}. Set
\begin{equation}\label{eq:wall-classes-d8-quadric}
  \vv_1\coloneqq (0,f^\ast f_1,-2),\qquad \vv_2\coloneqq (0,f^\ast f_2,-2)\in \widetilde H_{\mathrm{alg}}(S,\mathbb{Z}),
\end{equation}
so that \( \vv=(0, H, -8)=2\vv_1+2\vv_2=2(\vv_1+\vv_2) \in \widetilde H_{\mathrm{alg}}(S,\mathbb{Z})\) is not primitive.
Here \(\vv_1^2=\vv_2^2=0\), so both \(M_{\vv_1}(S,H)\) and \(M_{\vv_2}(S,H)\) are smooth K3 surfaces. The three irreducible components of the singular locus correspond to the minimal types
\[
  (1,\vv_1+\vv_2;1,\vv_1+\vv_2),\quad
  (1,\vv_1;1,\vv_1+2\vv_2),\quad
  (1,2\vv_1+\vv_2;1,\vv_2).
\]
The first component is structural: it is the \emph{generic} non-primitive singular locus, its wall class is zero, and it survives for every polarisation in \(\Amp(S)\). The other two components are genuine wall strata, and their wall classes define the same wall through \(H\). Thus the chamber-to-wall morphism resolves only the wall components and is a \emph{singular} terminalisation of \(M_\vv(S,H)\). This distinguishes the two cases in terms of their singularities. The Hasse diagrams associated with the polystable-type stratifications, which are intrinsic, provide another comparison; see \textsection\ref{subsec:stratification-polystable-type} and Figures~\ref{fig:singular-locus-Bl1P2} and~\ref{fig:singular-locus-P1xP1}.
\end{remark}

\section{The Prym involution on Beauville--Mukai systems}\label{sec:prym-involution-delPezzo}

This section introduces the Prym symplectic involution on the Beauville--Mukai system \(M_\vv(S,H)\) by composing the covering involution with a twisted derived duality. Over the invariant hyperplane in the support linear system \(f^\ast |-K_T|\subset |H|\), this is the usual Prym involution on the Jacobians of double covers of smooth genus-one anticanonical divisors in \(|-K_T|\). We describe its fixed locus, including the relative compactified Prym component of codimension \(2\) of~\cite{BrakkeeEtAl2026RelativePrym}, and prove that, for \(d\geq 2\), the induced action on second cohomology is trivial.

\subsection{The Prym symplectic involution}
The anti-symplectic covering involution \(\iota\) of \(f\colon S\to T\) defines, by pull-back, an involution of \(\Db(S)\), which we denote again by \(\iota\):
\begin{equation}
  \label{eq:iota-Db-delPezzo}
  \iota\colon \Db(S)\to \Db(S), \quad  \mathscr{F}^\bullet \mapsto \iota^\ast \mathscr{F}^\bullet.
\end{equation}
We also use the following contravariant involution, a twisted version of the derived dual:
\begin{equation}
  \label{eq:duality-Db-delPezzo}
  \mathbb{D}\colon \Db(S) \to \Db(S), \quad \mathscr{F}^\bullet \mapsto \mathbf{R}\mathcal{H}om(\mathscr{F}^\bullet, \mathscr{O}_S(-H))[1].
\end{equation}
The twist \(\mathscr{O}_S(-H)\) and the shift are chosen so that \(\mathbb{D}\) acts on \(M_{\vv}(S,H)\).

The polarisation \(H=-f^\ast K_T\in \NS(S)\) is \(\iota\)-equivariant, so \(\iota\) acts naturally on the linear system \(|H|\cong \mathbb{P}^{d+1}\).

\begin{lemma}\label{lem:iota-action-linear-system-delPezzo}
  The fixed locus of \(\iota\) on \(|H|\) consists of a hyperplane
  \begin{equation*}
    f^\ast |-K_T|\coloneqq \mathbb{P} f^\ast H^0(T, \mathscr{O}_T(-K_T))\cong \mathbb{P}^d
  \end{equation*}
  of \(\iota\)-invariant curves, and an isolated point corresponding to the smooth ramification curve \(R\subset S\). The locus of non-integral curves in \(|H|\) described in Proposition~\ref{prop:non-integral-members-delPezzo} is contained in the fixed hyperplane \(f^\ast |-K_T|\).
\end{lemma}

\begin{proof}
  For the first claim, decompose \(H^0(S, \mathscr{O}_S(H))\) into eigenspaces for the \(\iota\)-action
  \[ H^0(S, \mathscr{O}_S(H)) = f^\ast H^0(T, \mathscr{O}_T(-K_T)) \oplus \langle s_R \rangle, \]
  where \(s_R\in H^0(S, \mathscr{O}_S(H))\) is a section cutting out the ramification divisor \(R\), on which \(\iota\) acts as \(-1\). The second claim follows from Proposition~\ref{prop:non-integral-members-delPezzo} and the isomorphism
  \[ f^\ast \colon H^0(T, \mathscr{O}_T(D_i)) \iso H^0(S, \mathscr{O}_S(f^\ast D_i)) \]
  for any effective summand \(D_i\) of \(-K_T\) as in~\eqref{eq:pull-back-sections-delPezzo}.
\end{proof}

We now turn to the induced actions on Beauville--Mukai systems. The cohomological action of \(\mathbb{D}\) on the Mukai lattice is
\[
  \mathbb{D}(\mathbf{w})= - \mathbf{w}^\vee \cdot e^{-H} \in \widetilde H(S,\mathbb{Z}).
\]
In particular, for a Mukai vector \(\mathbf{w}\coloneqq (0,f^\ast D,\chi)\) with \(D\in \NS(T)\), we have
\[
  \mathbb{D}(\mathbf{w})=(0,f^\ast D,-\chi-f^\ast D\cdot H)=(0,f^\ast D,-\chi+2D\cdot K_T).
\]

The equivalence \(\mathbb{D}\) sends pure one-dimensional sheaves on \(S\) to pure one-dimensional sheaves~\cite[Prop.\@ 1.1.6]{HuybrechtsLehn2010GeometryModuliSheaves}. For such sheaves, it turns ordinary Gieseker stability into twisted Gieseker stability in the sense of~\cite[Def.\@ 3.2]{MatsukiWentworth1997MumfordThaddeusPrinciple}, because of the twist \(\mathscr{O}_S(-H)\). In some cases, after fixing the numerical class, this twisted stability condition can be rewritten as ordinary Gieseker stability for another polarisation:

\begin{lemma}\label{lem:duality-stability-reflection-delPezzo}
  Let \(D\in \NS(T)\) be a non-zero effective class, and \(\chi\in \mathbb{Z}\). Set \(\mathbf{w}\coloneqq (0,f^\ast D,\chi)\in \widetilde H_{\mathrm{alg}}(S,\mathbb{Z})\). For any ample class \(\alpha\in \Amp(S)\), the equivalence \(\mathbb{D}\) sends \(\alpha\)-semistable sheaves of Mukai vector \(\mathbf{w}\) to \(-H\)-twisted \(\alpha\)-semistable sheaves of Mukai vector \(\mathbb{D}(\mathbf{w})\), and vice versa.

  Assume moreover that \(\chi+f^\ast D\cdot H\neq 0\), and set
  \begin{equation}\label{eq:polarisation-transform-general}
    \beta\coloneqq \frac{\chi\alpha+(f^\ast D\cdot \alpha)H}{\chi+f^\ast D\cdot H}\in \NS(S)_{\mathbb{R}}.
  \end{equation}
  If \(\beta \in \Amp(S)\) is ample, then \(-H\)-twisted \(\alpha\)-semistability for sheaves of Mukai vector \(\mathbb{D}(\mathbf{w})\) is equivalent to ordinary \(\beta\)-semistability. Hence \(\mathbb{D}\) induces an isomorphism of good moduli spaces
  \[
    M_{\mathbf{w}}(S,\alpha)\iso M_{\mathbb{D}(\mathbf{w})}(S,\beta).
  \]
  The inverse is induced by \(\mathbb{D}\) again. For \(\alpha=H\), one has \(\beta=H\), and hence \(\beta\) is ample for all \(\alpha\) in a sufficiently small neighbourhood of \(H\).
\end{lemma}

\begin{proof}
  Let \(\mathscr{F}\) have Mukai vector \(\mathbf{w}\), and let \(\mathscr{F}\to \mathscr{G}\) be a pure one-dimensional quotient. The \(\alpha\)-semistability inequality for this quotient is
  \[
    \frac{\chi(\mathscr{G})}{c_1(\mathscr{G})\cdot\alpha}\geq \frac{\chi}{f^\ast D\cdot\alpha}.
  \]
  Since \(\mathbb{D}\) turns pure one-dimensional quotients into pure one-dimensional subsheaves, we have \(\mathbb{D}(\mathscr{G})\subset \mathbb{D}(\mathscr{F})\). The \(-H\)-twisted \(\alpha\)-semistability inequality for the subsheaf \(\mathbb{D}(\mathscr{G})\subset \mathbb{D}(\mathscr{F})\) is
  \[
    \frac{-\chi(\mathscr{G})}{c_1(\mathscr{G})\cdot\alpha}\leq \frac{-\chi}{f^\ast D\cdot\alpha},
  \]
  which is exactly the original \(\alpha\)-semistability inequality for \(\mathscr{F}\to \mathscr{G}\).

  Assume now that \(\chi+f^\ast D\cdot H\neq 0\), and let \(\beta\) in~\eqref{eq:polarisation-transform-general} be ample. Let \(\mathscr{F}\) have Mukai vector \(\mathbb{D}(\mathbf{w})=(0,f^\ast D,-\chi-f^\ast D\cdot H)\), and let \(\mathscr{G}\subset\mathscr{F}\) be a pure one-dimensional subsheaf. The \(\beta\)-semistability inequality for \(\mathscr{G}\subset\mathscr{F}\) is
  \[
    \frac{\chi(\mathscr{G})}{c_1(\mathscr{G})\cdot\beta}\leq \frac{-\chi-f^\ast D\cdot H}{f^\ast D\cdot\beta}.
  \]
  A straightforward substitution of~\eqref{eq:polarisation-transform-general} shows that this is equivalent to
  \[
    (f^\ast D\cdot\alpha)\chi(\mathscr{G})+\chi(c_1(\mathscr{G})\cdot\alpha)+(f^\ast D\cdot\alpha)(c_1(\mathscr{G})\cdot H)\leq 0,
  \]
  or equivalently
  \[
    (f^\ast D\cdot\alpha)(\chi(\mathscr{G})+c_1(\mathscr{G})\cdot H)+\chi(c_1(\mathscr{G})\cdot\alpha)\leq 0,
  \]
  which is precisely the \(-H\)-twisted \(\alpha\)-semistability inequality for \(\mathscr{G}\subset\mathscr{F}\). This proves the claimed equivalence of the two stability notions for sheaves of Mukai vector \(\mathbb{D}(\mathbf{w})\). Since \(\mathbb{D}^2\simeq \mathrm{id}\) on pure one-dimensional sheaves, this gives the claimed isomorphism of moduli stacks, and hence of good moduli spaces. The last statement follows from the fact that \(\beta\) depends continuously on \(\alpha\) and that \(\beta=H\) for \(\alpha=H\).
\end{proof}

We now restrict to Mukai vectors \(\mathbf{w}\coloneqq (0,f^\ast D,\chi)\in \widetilde{H}_\mathrm{alg}(S, \mathbb{Z})\) which are fixed by \(\mathbb{D}\). By the formula for \(\mathbb{D}(\mathbf{w})\), this is equivalent to \(\chi=D\cdot K_T\). In this case, the class \(\beta\) from~\eqref{eq:polarisation-transform-general} is
\[
  2\frac{f^\ast D\cdot \alpha}{f^\ast D\cdot H}H-\alpha.
\]
For the Mukai vector \(\vv\) of our Beauville--Mukai system, this takes the following simple form.

\begin{corollary}\label{cor:duality-reflection-delPezzo}
  For the Mukai vector \(\vv\), the transformation \(\alpha\mapsto\beta\) obtained from~\eqref{eq:polarisation-transform-general} is
  \begin{equation}\label{eq:polarisation-reflection-v}
    \beta=2\frac{\alpha\cdot H}{H^2}H-\alpha=\frac{\alpha\cdot H}{d}H-\alpha.
  \end{equation}
  It is the orthogonal reflection fixing \(\mathbb{R}H\) and acting as \(-1\) on \(H^\perp_{\mathbb{R}}\). Whenever \(\beta\) is ample, \(\mathbb{D}\) induces an isomorphism
  \[
    M_{\vv}(S,\alpha)\iso M_{\vv}(S,\beta).
  \]
\end{corollary}

\begin{proof}
  This is the preceding discussion with \(f^\ast D=H\), together with \(H^2=2d\). The isomorphism follows from Lemma~\ref{lem:duality-stability-reflection-delPezzo}.
\end{proof}

\begin{remark}[Twisted duality and the choice of polarisation]\label{rmk:forced-wall-prym-invariants}
  The above calculations explain why the duality involution naturally leads to a wall moduli space. Suppose more generally that we replace \(\mathscr{O}_S(-H)\) by an arbitrary line bundle \(\mathscr{L}\in \Pic(S)\), and set
  \[
    \mathbb{D}_{\mathscr{L}}\coloneqq \mathbf{R}\mathcal{H}om(-,\mathscr{L})[1],
    \qquad \ell\coloneqq c_1(\mathscr{L})\in \NS(S).
  \]
  For \(\mathbf{w}\coloneqq (0,H,\chi)\in \widetilde H_{\mathrm{alg}}(S,\mathbb{Z})\), the numerical condition that \(\mathbb{D}_{\mathscr{L}}\) preserve \(\mathbf{w}\) is \(2\chi=H\cdot \ell\). If this holds and \(\chi-H\cdot \ell\neq 0\), then, for any ample class \(\alpha\in \Amp(S)\), Lemma~\ref{lem:duality-stability-reflection-delPezzo} identifies the transformed stability condition with ordinary stability for
  \[
    \beta\coloneqq \frac{\chi\alpha-(H\cdot \alpha)\ell}{\chi-H\cdot\ell}\in \NS(S)_{\mathbb{R}}
  \]
  whenever \(\beta\) is ample. Thus, if one wants \(\mathbb{D}_{\mathscr{L}}\) to act naturally on the chamber moduli problem \(M_{\mathbf{w}}(S,\alpha)\), rather than identify it with a different chamber model, the relevant condition is that \(\alpha\) and \(\beta\) lie in the same \(\mathbf{w}\)-chamber.

  In the del Pezzo cases with \(2\leq d\leq 8\), this fails systematically. Assume \(\chi\neq 0\), otherwise there is no \(\mathbf{w}\)-generic polarisation. Then
  \[
    \beta=2\frac{H\cdot \alpha}{H\cdot \ell}\,\ell-\alpha.
  \]
  For an admissible splitting \(-K_T=D_1+D_2\) as in Remark~\ref{rmk:explicit-minimal-polystable-strata-delPezzo}, the Mukai vectors
  \[
    \mathbf{w}_i\coloneqq \left(0,f^\ast D_i,\frac{1}{2}(f^\ast D_i\cdot \ell)\right), \qquad i=1,2,
  \]
  have non-empty summand moduli spaces, and the associated wall class is, up to sign,
  \[
    \delta_{D_1, D_2}\coloneqq (f^\ast D_1\cdot \ell)f^\ast D_2-(f^\ast D_2\cdot \ell)f^\ast D_1\in \NS(S).
  \]
  Since \(\delta_{D_1, D_2}\cdot \ell=0\), we have \(\delta_{D_1, D_2}\cdot \beta=-\delta_{D_1, D_2}\cdot \alpha\). Thus, whenever \(\alpha\) is \(\mathbf{w}\)-generic and \(\beta\) is ample, the two polarisations lie on opposite sides of an actual wall. In other words, \(\mathbb{D}_{\mathscr{L}}\) identifies the two chamber models across the corresponding wall, but it does not preserve the stability condition on either chamber. Passing to the wall polarisation is what makes the duality descend to an involution of a single moduli space. The same behaviour occurs in the unramified setting of~\cite[\textsection3.5, Prop.\@ 3.16]{ArbarelloEtAl2015RelativePrymEnriques}.
\end{remark}

We can now describe the symplectic involution on \(M_{\vv}(S,H)\) induced by the composition of \(\iota\) and \(\mathbb{D}\):

\begin{proposition}
  \label{prop:involutions-delPezzo}
  The equivalences \(\iota\) and \(\mathbb{D}\) define anti-symplectic commuting involutions on \(M_{\vv}(S,H)\). Their composition
  \begin{equation}
    \label{eq:tau-involution-delPezzo}
    \tau\colon \Db(S)\to \Db(S), \quad \mathscr{F}^\bullet \mapsto \mathbf{R}\mathcal{H}om(\iota^\ast \mathscr{F}^\bullet, \mathscr{O}_S(-H))[1]
  \end{equation}
  defines a non-trivial symplectic involution on \(M_{\vv}(S,H)\). The Lagrangian fibration
  \[ \pi\colon M_{\vv}(S,H) \to |H| \]
  is equivariant for the actions of \(\tau\) on the source and \(\iota\) on the target, and the singular locus of \(M_{\vv}(S,H)\) is pointwise fixed by \(\tau\).
\end{proposition}

The involution \(\tau\) is a \emph{Prym involution} in the following sense. Over the \(\iota\)-fixed hyperplane \(f^\ast|-K_T|\subset |H|\), a general curve \(C\) is a smooth genus \(d+1\) double cover of a smooth genus-one curve \(E\in |-K_T|\), with covering involution given by \(\iota\). On the fibre \(\pi^{-1}(C)\cong \Pic^0(C)\), the involution \(\tau\) acts as
\begin{equation}\label{eq:tau-Prym-delPezzo}
  \mathscr{L}\mapsto \iota^\ast \mathscr{L}^\vee,
\end{equation}
whose fixed locus is the Prym variety of the double cover \(f\colon C\to E\). This is the involution used in~\cite{MarkushevichTikhomirov2007Prymian, ArbarelloEtAl2015RelativePrymEnriques, Matteini2016SingularPrymFibration, BrakkeeEtAl2026RelativePrym} to construct relative compactified Prym varieties from K3 surfaces equipped with anti-symplectic involutions.

\smallskip
The first special feature of the del Pezzo setting highlighted by Proposition~\ref{prop:involutions-delPezzo} is that the singular locus of \(M_{\vv}(S,H)\) is pointwise fixed by \(\tau\). This is the hyperelliptic analogue of the Prym picture above: the Jordan--H\"older factors of a general singular point are supported on smooth hyperelliptic double covers of rational curves in \(T\). In this case, the corresponding Prym variety is the whole Jacobian, or equivalently the action~\eqref{eq:tau-Prym-delPezzo} is trivial. This is the key input in the proof:

\begin{proof}[Proof of Proposition~\ref{prop:involutions-delPezzo}]
  Any Mukai vector \(\mathbf{w}\coloneqq (0, f^\ast D, \chi)\in \widetilde H_{\mathrm{alg}}(S,\mathbb{Z})\) is fixed by the action of \(\iota\). Since \(H=-f^\ast K_T\) is also \(\iota\)-invariant, \(\iota\) induces an anti-symplectic involution on the associated Beauville--Mukai moduli space \(M_{\mathbf{w}}(S,H)\)~\cite[Lemma 3.6]{ArbarelloEtAl2015RelativePrymEnriques}. By Lemma~\ref{lem:duality-stability-reflection-delPezzo}, the duality \(\mathbb{D}\) preserves \(\mathbf{w}\) and \(H\)-semistability precisely when \(\chi=D\cdot K_T\); in that case it induces an anti-symplectic involution on \(M_{\mathbf{w}}(S,H)\)~\cite[Prop.\@ 3.11]{ArbarelloEtAl2015RelativePrymEnriques}. In particular, this holds for our choice of Mukai vector \(\vv=(0, H, -d)\in \widetilde H_{\mathrm{alg}}(S,\mathbb{Z})\), so that \(\tau\) is a well-defined symplectic involution on \(M_{\vv}(S,H)\).

  Since \(\mathbb{D}\) does not affect the support of a sheaf, \(\pi\) is equivariant for the actions of \(\tau\) and \(\iota\). Since the action of \(\iota\) is non-trivial on \(|H|\) and \(\pi\) is surjective, the action of \(\tau\) on \(M_{\vv}(S,H)\) is non-trivial.

  For the last claim, by \textsection\ref{subsec:stratification-polystable-type}, the singular locus of \(M_{\vv}(S,H)\) is the union of the closures of the strata \(M_{\vv}^\mathfrak{t}(S,H)\) described in Proposition~\ref{prop:polystable-strata-delPezzo}. For any such decomposition \(\vv=\vv_1+\vv_2\), the action of \(\tau\) is well-defined on each \(M_{\vv_i}(S,H)\), and the direct sum morphism~\eqref{eq:direct-sum-morphism} is \(\tau\)-equivariant. It is therefore enough to check that the action of \(\tau\) on \(M_{\vv_i}(S,H)\) is trivial for any \(\vv_i\) as in Proposition~\ref{prop:polystable-strata-delPezzo}.

  Let \(D\in \NS(T)\) be an effective direct summand of the anticanonical class \(-K_T\) such that \(D+K_T\) is anti-effective, and assume that a general member \(\Gamma\) of \(|D|\) is a smooth rational curve, according to Proposition~\ref{prop:non-integral-members-delPezzo}. The corresponding general member of \(|f^\ast D|\) on \(S\) is a \(\iota\)-invariant smooth hyperelliptic curve \(C\) of genus \(D^2+1\), which is a double cover \(f\colon C\to \Gamma\) branched at \(-2 \Gamma\cdot K_T=2D^2+4\) points. The covering involution \(\iota\) restricts to the hyperelliptic involution on \(C\).

  Now let
  \[ \mathbf{w}\coloneqq (0, f^\ast D, D\cdot K_T)\in \widetilde H_{\mathrm{alg}}(S,\mathbb{Z}) \]
  be the corresponding direct summand of \(\vv\). Here the \(\iota\)-action on \(|f^\ast D|\) is trivial by the pull-back isomorphism~\eqref{eq:pull-back-sections-delPezzo}, so the moduli space \(M_{\mathbf{w}}(S,H)\) comes with a \(\tau\)-invariant Lagrangian fibration
  \[ M_{\mathbf{w}}(S,H)\to |f^\ast D|\cong \mathbb{P}^{D^2+1}, \]
  whose fibre over \(C\) is the component of the Picard scheme \(\Pic^{-2}(C)\). Let \(i\colon C\hookrightarrow S\) be the inclusion. For any line bundle \(\mathscr{L}\in \Pic^{-2}(C)\), the actions of \(\iota\) and \(\mathbb{D}\) on \(i_\ast \mathscr{L}\) are given by
  \[ \iota(i_\ast \mathscr{L})=i_\ast \iota^\ast \mathscr{L}\cong i_\ast(\mathscr{L}^\vee \otimes f^\ast \mathscr{O}_\Gamma (-2)), \quad \mathbb{D}(i_\ast \mathscr{L})=i_\ast(\mathscr{L}^\vee \otimes \mathscr{O}_C(f^\ast D-H)). \]
  Since \(\mathscr{O}_C(f^\ast D-H)\cong f^\ast \mathscr{O}_\Gamma(D+K_T)\) and \(\mathscr{O}_\Gamma(D+K_T)\cong \mathscr{O}_\Gamma(-2)\) by adjunction, we deduce that \(\iota(i_\ast \mathscr{L})=\mathbb{D}(i_\ast \mathscr{L})\) for all \(\mathscr{L}\in \Pic^{-2}(C)\). Since \(\tau\) acts trivially on the fibre over a general curve in \(|f^\ast D|\), it acts trivially on \(M_{\mathbf{w}}(S,H)\), as claimed.
\end{proof}

\begin{remark}[Action on \(\vv\)-chambers and Mukai morphisms]\label{rmk:tau-chambers-mukai-morphisms}
Assume that \(\alpha\in \Amp(S)\) is \(\vv\)-generic and sufficiently close to \(H\), and that the class \(\beta\) from~\eqref{eq:polarisation-reflection-v} is also ample and \(\vv\)-generic. Since \(\NS(S)=f^\ast\NS(T)\), the involution \(\iota\) fixes \(\alpha\) and \(\beta\). Hence Corollary~\ref{cor:duality-reflection-delPezzo} gives an isomorphism
  \[
    \tau\colon M_\vv(S,\alpha)\iso M_\vv(S,\beta).
  \]
On Mukai vectors of objects, this anti-autoequivalence acts by
\[
  \mathbf{w}\mapsto -\iota^\ast(\mathbf{w}^\vee\cdot e^{-H})\in \widetilde H(S,\mathbb{Z}).
\]
Let
\[
  \lambda_\vv^\alpha\colon \vv^\perp\iso H^2(M_\vv(S,\alpha),\mathbb{Z})
\]
be the Mukai morphism of Theorem~\ref{thm:2nd-cohomology-moduli-sheaves}, and let \(\lambda_\vv^\beta\) denote the corresponding morphism for \(M_\vv(S,\beta)\). Yoshioka's functoriality of the Mukai morphism under anti-autoequivalences gives an additional sign~\cite[Prop.\@ 2.5]{Yoshioka2001ModuliAbelianSurfaces} (see also~\cite[Prop.\@ 3.10]{Bottini2024StableSheavesK3Surfaces}), so that we have a commutative diagram
\begin{equation}\label{eq:action-tau-vperp-Mukai-morphism}
  \begin{tikzcd}
    \vv^\perp \arrow[r, "\lambda_\vv^\beta"] \arrow[d, "\iota^\ast((-)^\vee\cdot e^{-H})"] & H^2(M_\vv(S, \beta), \mathbb{Z}) \arrow[d, "\tau^\ast"] \\
    \vv^\perp \arrow[r, "\lambda_\vv^\alpha"] & H^2(M_\vv(S, \alpha), \mathbb{Z})
  \end{tikzcd}
\end{equation}
or equivalently
\[
  \tau^\ast\lambda_\vv^\beta(\mathbf{w})
  =
  \lambda_\vv^\alpha\bigl(\iota^\ast(\mathbf{w}^\vee\cdot e^{-H})\bigr),
  \qquad \mathbf{w}\in \vv^\perp.
\]
\end{remark}

\subsection{The fixed locus and relative Prym varieties}
We now investigate the fixed locus of \(\tau\) on \(M_{\vv}(S,H)\). It contains a unique irreducible component of codimension \(2\), namely the relative compactified Prym variety studied systematically in~\cite{BrakkeeEtAl2026RelativePrym}. This component is the source of the strictly canonical singularities in the quotient, and hence of the non-trivial terminalisation. 

\begin{proposition}
  \label{prop:fixed-locus-tau-delPezzo}
  The fixed locus \(M_{\vv}^\tau(S,H)\) contains:
  \begin{enumerate}
    \item a unique irreducible component of codimension \(2\), equipped with a Lagrangian fibration
    \begin{equation} \label{eq:lagrangian-fibration-Prym}
      \pi\colon P_{\vv}(S,H)\to f^\ast|-K_T|\subset |H|^\iota
    \end{equation}
    which contains the singular locus \(M_{\vv}^{\mathrm{sing}}(S,H)\) and whose general fibre is a Prym variety of a double cover of a smooth genus-one curve;
    \item finitely many isolated points given by the \(2\)-torsion \(\Pic^0(R)[2]\) over the smooth ramification curve \(R\in |H|\). 
  \end{enumerate} 
\end{proposition}

The normalisation of the maximal fixed component \(P_{\vv}(S,H)\subset M^\tau_{\vv}(S,H)\) is the \emph{relative Prym variety} \(\mathcal{P}_H\) of~\cite[Def.\@ 2.8]{BrakkeeEtAl2026RelativePrym}. For \(d\geq 3\), \(\mathcal{P}_H\) is shown to be an irreducible symplectic variety~\cite[Cor.\@ 7.7]{BrakkeeEtAl2026RelativePrym}. 

The following lemma provides the local quiver computation used in the proof of Proposition~\ref{prop:fixed-locus-tau-delPezzo}.

\begin{lemma}\label{lem:Kronecker-quiver-tau}
  Let \(n\geq 1\) and \(K\) be the \(2n\)-Kronecker quiver with two vertices and \(2n\) arrows between them. Let
  \begin{equation}\label{eq:tau-skew-symmetric-isomorphism}
    \tau \colon \mathrm{Rep}(K, (1,1))\iso \mathrm{Rep}(K^{\mathrm{op}}, (1,1)).
  \end{equation}
  be a skew-symmetric isomorphism, and denote by \(\tau\) also the induced symplectic involution on \(\mathrm{Rep}(\overline{K}, (1,1))\).
  Then there are isomorphisms
  \begin{equation*}
    \mathfrak{M}(K, (1,1)) \cong \overline{\mathcal{O}}_{\mathrm{min}}^{\mathfrak{sl}_{2n}},
    \qquad
    \mathfrak{M}^\tau(K, (1,1)) \cong \overline{\mathcal{O}}_{\mathrm{min}}^{\mathfrak{sp}_{2n}}.
  \end{equation*}
  More generally, if the quiver \(Q\) is obtained from \(K\) by adjoining \(m_1\) and \(m_2\) loops at the two vertices and \(\tau\) acts on \(\mathrm{Rep}(\overline{Q}, (1,1))\) by the identity on the loop spaces, then
  \begin{equation*}
    \mathfrak{M}(Q, (1,1)) \cong \overline{\mathcal{O}}_{\mathrm{min}}^{\mathfrak{sl}_{2n}} \times \mathbb{C}^{2m_1+2m_2},
    \qquad
    \mathfrak{M}^\tau(Q, (1,1)) \cong \overline{\mathcal{O}}_{\mathrm{min}}^{\mathfrak{sp}_{2n}} \times \mathbb{C}^{2m_1+2m_2}.
  \end{equation*}
\end{lemma}

Here \(\overline{\mathcal{O}}^\mathfrak{g}_{\mathrm{min}}\subset \mathfrak{g}\) denotes the closure of the minimal nilpotent orbit in the simple Lie algebra \(\mathfrak{g}\). In particular, it is irreducible and normal~\cite{KraftProcesi1982ConjugacyClasses}. 

\begin{proof}
  Set \(V\coloneqq \mathrm{Rep}(K, (1,1))\). Then \(V^\vee\cong \mathrm{Rep}(K^{\mathrm{op}}, (1,1))\) and \(\dim V=2n\). The group \(\mathbb{P}G\coloneqq \mathbb{G}_m\) acts on
  \[
    \mathrm{Rep}(\overline{K}, (1,1))\cong V\oplus V^\vee
  \]
  with weights \((1,-1)\), and, up to a scalar, a moment map is given by
  \[
    \mu\colon V\oplus V^\vee \to \mathrm{Lie}(\mathbb{P}G)^\vee \cong \mathbb{C},\qquad (v,\alpha)\mapsto \alpha(v).
  \]
  The quotient morphism
  \begin{equation}\label{eq:quotient-Kronecker-minimal-orbit}
    \mu^{-1}(0)\to \End(V), \qquad (v,\alpha)\mapsto v\otimes \alpha
  \end{equation}
  has image the cone of traceless endomorphisms of rank at most \(1\), namely \(\overline{\mathcal{O}}_{\mathrm{min}}^{\mathfrak{sl}(V)}\cong \overline{\mathcal{O}}_{\mathrm{min}}^{\mathfrak{sl}_{2n}}\). Hence \(\mathfrak{M}(K, (1,1))\cong \overline{\mathcal{O}}_{\mathrm{min}}^{\mathfrak{sl}_{2n}}\). The skew-symmetric isomorphism~\eqref{eq:tau-skew-symmetric-isomorphism} induces the involution
  \begin{equation} \label{eq:involution-quiver-local}
    (v, \alpha)\mapsto (\tau^{-1}(\alpha), \tau(v))
  \end{equation}
  on \(\mathrm{Rep}(\overline{K}, (1,1))\). Choosing symplectic coordinates on \(V\), so that \(\tau\) is represented by the standard symplectic form \(J\), the morphism~\eqref{eq:quotient-Kronecker-minimal-orbit} identifies \eqref{eq:involution-quiver-local} with the involution
  \[
    A\mapsto -JA^tJ^{-1}
  \]
  on \(\overline{\mathcal{O}}_{\mathrm{min}}^{\mathfrak{sl}_{2n}}\). Its fixed locus is the intersection with \(\mathfrak{sp}_{2n}\), hence \(\mathfrak{M}^\tau(K, (1,1))\cong \overline{\mathcal{O}}_{\mathrm{min}}^{\mathfrak{sp}_{2n}}\). 
  
  For the quiver \(Q\), the added loops contribute a weight-zero symplectic summand of dimension \(2m_1+2m_2\), so the symplectic reduction splits as a product with \(\mathbb{C}^{2m_1+2m_2}\). Since \(\tau\) acts trivially on this summand by assumption, the same product decomposition holds for the fixed locus. This proves the second claim.
\end{proof}

\begin{proof}[Proof of Proposition~\ref{prop:fixed-locus-tau-delPezzo}]
  The uniqueness of the relative Prym component \(P_{\vv}(S,H)\), together with its properties, is established in~\cite[\textsection2.2.2, Cor.\@ 2.16]{BrakkeeEtAl2026RelativePrym}. It remains to show that it contains the whole singular locus of \(M_{\vv}(S,H)\), which is pointwise fixed by Proposition~\ref{prop:involutions-delPezzo}. By uniqueness, it is enough to show that a general singular point in each of the minimal strata \(M^\mathfrak{t}_\vv(S,H)\) described in Proposition~\ref{prop:polystable-strata-delPezzo} lies in a codimension-\(2\) component of the fixed locus.  

  Let \(\mathfrak{t}\coloneqq (1, \vv_1; 1, \vv_2)\) be such a minimal polystable type, and let \([\mathscr{F}]\in M^\mathfrak{t}_\vv(S,H)\) be a general point, with
  \[ \mathscr{F}=\mathscr{F}_1\oplus \mathscr{F}_2, \]
  where \(\mathscr{F}_1\) and \(\mathscr{F}_2\) are non-isomorphic stable sheaves with Mukai vectors \(\vv_1\) and \(\vv_2\), respectively. Set
  \[ V_{ij}\coloneqq \Ext^1(\mathscr{F}_i, \mathscr{F}_j), \qquad 1\leq i,j\leq 2. \]
  Since \(\mathfrak{t}\) is minimal, we have \(\dim V_{12}=\dim V_{21}=\vv_1\cdot \vv_2=4\), while \(\dim V_{ii}=\vv_i^2+2\). By the local description from \textsection\ref{subsec:Kuranishi-local-model}, the \'etale-local germ of \(M_{\vv}(S,H)\) at \([\mathscr{F}]\) is identified with the affine symplectic reduction
  \begin{equation*}
    \Ext^1(\mathscr{F}, \mathscr{F}) \sympq \mathbb{P}G,
  \end{equation*}
  where \(\mathbb{P}G\cong \mathbb{G}_m\) and
  \begin{equation*}
    \Ext^1(\mathscr{F}, \mathscr{F}) \cong V_{11}\oplus V_{12}\oplus V_{21}\oplus V_{22}
  \end{equation*}
  with weights \((0,1,-1,0)\), respectively.

  Since the action of \(\tau\) is trivial on each \(M_{\vv_i}(S,H)\) by Proposition~\ref{prop:involutions-delPezzo}, we can choose isomorphisms \(\alpha_i\colon \mathscr{F}_i\iso \tau(\mathscr{F}_i)\). Since \(\tau\) is contravariant, this yields linear isomorphisms
  \begin{equation*}
    \tau_{ij}\colon V_{ij}\iso V_{ji}, \qquad e\mapsto \alpha_i^{-1}\circ \tau(e)\circ \alpha_j.
  \end{equation*}
  For \(i=j\), the tangent space \(T_{[\mathscr{F}_i]}M_{\vv_i}(S,H)\) is identified with \(V_{ii}\), and since \(\tau\) acts trivially on \(M_{\vv_i}(S,H)\), its differential at \([\mathscr{F}_i]\) is the identity. Thus \(\tau\) acts trivially on \(V_{11}\oplus V_{22}\). On the other hand, Serre duality identifies \(V_{21}\cong V_{12}^\vee\), and because \(\tau\) is symplectic on \(M_{\vv}(S,H)\), the map \(\tau_{12}\colon V_{12}\to V_{21}\cong V_{12}^\vee\) is skew-symmetric with respect to this pairing. Therefore the \(\tau\)-equivariant local model splits as
  \begin{equation*}
    (V_{11}\oplus V_{22}) \times (V_{12}\oplus V_{12}^\vee)\sympq \mathbb{G}_m,
  \end{equation*}
  where \(\tau\) acts trivially on the first factor and on the second factor via the involution induced by \(\tau_{12}\). Equivalently, if \(Q\) denotes the corresponding Ext-quiver, then \(Q\) is obtained from the \(4\)-Kronecker quiver \(K\) by adjoining \(m_i\coloneqq \frac{1}{2}(\vv_i^2+2)\) loops at the vertex \(i\), and the induced \(\tau\)-action on \(\mathrm{Rep}(\overline{Q}, (1,1))\) is precisely the one appearing in Lemma~\ref{lem:Kronecker-quiver-tau}: it is induced by the skew-symmetric map \(\tau_{12}\) on the Kronecker part and acts as the identity on the loop spaces. By Lemma~\ref{lem:Kronecker-quiver-tau}, the \(\tau\)-fixed locus in the local model of \(M_{\vv}(S,H)\) at \([\mathscr{F}]\) is isomorphic to
  \begin{equation*}
    \overline{\mathcal{O}}_{\mathrm{min}}^{\mathfrak{sp}_4}\times \mathbb{C}^{2m_1+2m_2}.
  \end{equation*}
  Since
  \begin{equation*}
    2m_1+2m_2=\vv_1^2+\vv_2^2+4=\vv^2-2\vv_1\cdot \vv_2+4=2d-4,
  \end{equation*}
  this fixed locus is irreducible of dimension \(4+(2d-4)=2d\), and therefore gives a single irreducible codimension \(2\) component of \(M^\tau_{\vv}(S,H)\) through the general point \([\mathscr{F}]\).
  
  To prove \(\mathrm{(ii)}\), recall from Proposition~\ref{prop:involutions-delPezzo} that the support morphism \(\pi\colon M_{\vv}(S,H)\to |H|\) is \(\tau\)-equivariant with respect to the action of \(\iota\) on \(|H|\). By Lemma~\ref{lem:iota-action-linear-system-delPezzo}, the smooth ramification curve \(R\subset S\) gives an isolated fixed point of \(\iota\) on \(|H|\), and the fibre over \(R\) is the Jacobian \(\pi^{-1}(R)\cong \Pic^0(R)\). On this fibre, the involution is given by~\eqref{eq:tau-Prym-delPezzo}; since \(\iota\) fixes \(R\) pointwise, it reduces to \(\mathscr{L}\mapsto \mathscr{L}^\vee\). Hence the fixed locus in \(\pi^{-1}(R)\) is precisely the finite subgroup \(\Pic^0(R)[2]\), and these are isolated points of \(M^\tau_{\vv}(S,H)\).
\end{proof}

\begin{remark}\label{rmk:fixed-locus-exhaustive} The list of connected components of the fixed locus \(M_{\vv}^\tau(S,H)\) given in Proposition~\ref{prop:fixed-locus-tau-delPezzo} is not \emph{a priori} complete. There may be other components of codimension \(\geq 4\). For \(d=1\), \(M_\vv(S,H)\) is a smooth fourfold of \(\mathrm{K3}^{[2]}\)-type, and by~\cite[Thm.\@ 4.1]{Mongardi2012SymplecticInvolutions} the fixed locus consists of a smooth K3 surface (the relative Prym variety) and 28 isolated points. Here \(\Pic^0(R)[2]\) contains only 16 points, while the other points live over the 12 singular members of the anticanonical pencil \(|-K_T|\). For \(d=2\), so \(\dim M_\vv(S,H)=6\), it is claimed in~\cite[\textsection3]{MarkushevichTikhomirov2007Prymian} that \(M^\tau_\vv(S,H)\) consists of the \(4\)-dimensional component \(P_{\vv}(S,H)\)---containing the 28 isolated singular points of \(M^\tau_\vv(S,H)\) listed in Remark~\ref{rmk:explicit-minimal-polystable-strata-delPezzo}---and the 64 isolated points \(\Pic^0(R)[2]\). It would be interesting to know whether the above list is exhaustive as soon as \(|-K_T|\) is base-point-free.
\end{remark}

\subsection{Induced action on cohomology}
The second special feature of the del Pezzo setting is that the Prym involution acts trivially on the second cohomology; this is Theorem~\ref{thm:intro-trivial-action-cohomology} from the Introduction:

\begin{theorem}\label{thm:trivial-action-cohomology} 
  For \(d\geq 2\), the Prym involution \(\tau\) acts trivially on \(H^2(M_\vv(S, H), \mathbb{Z})\).
\end{theorem}

The existence of such symplectic automorphisms acting trivially on cohomology is useful for two reasons. First, to the best of our knowledge, it had not previously been recorded for singular moduli spaces of sheaves, or more generally for singular irreducible symplectic varieties. Second, it is crucial for the quotient construction below: passing to the quotient by \(\tau\) does not kill any degree-\(2\) classes. This is what allows the terminalisations below to retain high \(b_2\). By contrast, for finite symplectic quotients of smooth hyper-K\"ahler manifolds of \(\mathrm{K3}^{[n]}\)-type, faithfulness on \(H^2\) forces a strict drop in the invariant part of the second cohomology; see also Remark~\ref{rmk:trivial-action-cohomology-K3n} and compare with~\cite{BertiniEtAl2025TerminalizationsQuotients}.

\begin{proof}
  By Proposition~\ref{prop:geometry-Mv-delPezzo}, \(M_\vv(S, H)\) is an irreducible symplectic variety. Hence its rational second cohomology carries the pure weight-two Hodge structure recalled in Section~\ref{sec:zoo-symplectic-varieties}. The decomposition into \(\pm 1\)-eigenspaces for \(\tau^\ast\)
  \[
    H^2(M_\vv(S, H),\mathbb{Q})=H^2(M_\vv(S, H),\mathbb{Q})^+ \oplus H^2(M_\vv(S, H),\mathbb{Q})^-
  \]
  is a decomposition by rational Hodge substructures. Since \(\tau\) is symplectic, it acts trivially on \(H^{2,0}(M_\vv(S,H))\), hence the \((-1)\)-eigenspace has no \((2,0)\)- or \((0,2)\)-part. Therefore
  \[
    H^2(M_\vv(S,H),\mathbb{Q})^- \subset H^{1,1}(M_\vv(S,H))\cap H^2(M_\vv(S,H),\mathbb{Q})=\NS(M_\vv(S,H))_{\mathbb{Q}}.
  \]

  On the other hand, \(\tau\) fixes \(\NS(M_\vv(S,H))_{\mathbb{Q}}\) pointwise. Indeed, the support morphism \(\pi\colon M_\vv(S,H)\to |H|\) is \(\tau\)-equivariant by Proposition~\ref{prop:involutions-delPezzo}, hence the class of the Lagrangian fibration is \(\tau\)-invariant. Averaging any ample class gives a \(\tau\)-invariant ample class. These two classes are linearly independent. As \(\rho(M_\vv(S,H))=2\), they span \(\NS(M_\vv(S,H))_{\mathbb{Q}}\). Hence \(H^2(M_\vv(S,H),\mathbb{Q})^-=0\), and \(\tau\) acts trivially on \(H^2(M_\vv(S,H),\mathbb{Q})\). Since \(H^2(M_\vv(S,H),\mathbb{Z})\) is torsion-free for irreducible symplectic varieties, it follows that \(\tau\) acts trivially on \(H^2(M_\vv(S,H),\mathbb{Z})\).
\end{proof}

\begin{remark}\label{rmk:trivial-action-cohomology-K3n}
  The assumption \(d\geq 2\) is essential. For \(d=1\), Proposition~\ref{prop:geometry-Mv-delPezzo} shows that \(M_\vv(S,H)\) is a smooth hyper-K\"ahler fourfold of \(\mathrm{K3}^{[2]}\)-type, and more generally the natural representation
  \[
    \mathrm{Aut}(Y)\to \mathrm{GL}(H^2(Y,\mathbb{Z}))
  \]
  is faithful for every hyper-K\"ahler manifold \(Y\) of \(\mathrm{K3}^{[n]}\)-type by~\cite[Lemma 1.2]{Mongardi2013NaturalDeformations}. In particular, the phenomenon of Theorem~\ref{thm:trivial-action-cohomology} cannot occur in the smooth \(\mathrm{K3}^{[n]}\)-type case. On the other hand, faithfulness on the second cohomology fails for generalised Kummer varieties and for manifolds of \(\mathrm{OG}6\)-type: the kernel for a generalised Kummer variety was computed in~\cite[Thm.\@ 3, Cor.\@ 5(2)]{BoissiereNieperWisskirchenSarti2011HigherDimensionalEnriques}, while for \(\mathrm{OG}6\)-type manifolds it is isomorphic to \((\mathbb{Z}/2\mathbb{Z})^8\) by~\cite{MongardiWandel2017AutomorphismsOGrady}. 
  
  Floccari's hyper-Kummer construction of finite quotients of hyper-K\"ahler sixfolds of \(\mathrm{Kum}^3\)-type admitting crepant resolutions also involves a finite group of symplectic automorphisms acting trivially on \(H^2(K,\mathbb{Z})\)~\cite[Thm.\@ 1.1]{Floccari2024Sixfolds}. Thus the two main features of our examples, namely codimension-\(2\) components in the fixed locus, which allow for a non-trivial terminalisation, and the trivial action on the second cohomology, also occur in their setting.
\end{remark}

\section{The quotient and its terminalisation}\label{sec:terminalisation-quotient}

We now describe the quotient of \(M_\vv(S,H)\) by the Prym involution \(\tau\) and its terminal models. The two facts proved above, namely the existence of the codimension-\(2\) Prym component \(P_{\vv}(S,H)\) in the fixed locus and the trivial action on second cohomology, control both the generic singularities of the quotient and the second Betti number of its terminalisations.

\subsection{Global geometry of the quotient}

We write \(X\coloneqq M_{\vv}(S,H)/\langle \tau \rangle\) for the quotient by the symplectic involution \(\tau\), and \(q\colon M_{\vv}(S,H)\to X\) for the quotient morphism. By Proposition~\ref{prop:involutions-delPezzo}, the Lagrangian fibration \(\pi\colon M_{\vv}(S,H)\to |H|\) induces a morphism \(\pi\colon X \to |H|/\langle \iota \rangle\), and by Lemma~\ref{lem:iota-action-linear-system-delPezzo}, the quotient \(|H|/\langle \iota \rangle\) is a weighted projective space \(\mathbb{P}(1^{d+1}, 2)\).

\begin{proposition}
\label{prop:geometry-X}
  The quotient \(X\) is a \(\mathbb{Q}\)-factorial irreducible symplectic variety of dimension \(2d+2\) with 
  \[ b_2(X) = 14+d, \quad \rho(X)= 2. \]
  It is endowed with a Lagrangian fibration to a weighted projective space
  \begin{equation*}
    \pi\colon X\to \mathbb{P}(1^{d+1}, 2).
  \end{equation*}
  Its reduced singular locus contains a unique irreducible component \(q(P_{\vv}(S,H))\) of codimension \(2\), lying over the smooth divisor \(f^\ast|-K_T| \cong \mathbb{P}^d \subset \mathbb{P}(1^{d+1}, 2)\), and at least \(2^{2d+2}\) isolated points lying over the singular point of \(\mathbb{P}(1^{d+1}, 2)\).
\end{proposition}

In particular, note that the base $|H|/\langle \iota \rangle$ is a quotient of \(\mathbb{P}^{d+1}\), as predicted by~\cite[Conj.\@ 1.4]{LiuLiuXu2025IrreducibleSymplecticVarieties} for Lagrangian fibrations on projective irreducible symplectic varieties.

As in Remark~\ref{rmk:fixed-locus-exhaustive}, we do not claim that the above description exhausts the singular locus of \(X\). This does not affect the computation of \(b_2\) for a \(\mathbb{Q}\)-factorial terminalisation, since only codimension-\(2\) components contribute exceptional divisors. The isolated quotient singularities listed above will be used only to ensure that the terminalisation is not smooth. 

\begin{proof} 
  Since \(M_{\vv}(S,H)\) is an irreducible symplectic variety by Proposition~\ref{prop:geometry-Mv-delPezzo} and \(\tau\) is symplectic, the finite quotient \(X\) is again an irreducible symplectic variety by Proposition~\ref{prop:criteria-ISV}. The equivariance of the support morphism in Proposition~\ref{prop:involutions-delPezzo} gives the Lagrangian fibration
  \[
    \pi\colon X\to |H|/\langle\iota\rangle\cong \mathbb{P}(1^{d+1},2).
  \]

  We next prove \(\mathbb{Q}\)-factoriality. If \(d=1\), then \(M_{\vv}(S,H)\) is smooth; if \(d=9\), then \(M_{\vv}(S,H)\) is locally factorial by Proposition~\ref{prop:geometry-Mv-delPezzo}. In both cases the finite quotient is \(\mathbb{Q}\)-factorial. Assume therefore that \(2\leq d\leq 8\). By~\cite[Lemma 1.1.3]{Prokhorov2021EquivariantMMP}, it suffices to show that the injective linear map
  \[
    \Pic(M_{\vv}(S,H))_{\mathbb{Q}}^{\tau}\to \Cl(M_{\vv}(S,H))_{\mathbb{Q}}^{\tau}
  \]
  is an isomorphism, or equivalently that
  \begin{equation}\label{eq:Q-factoriality-criterion-delPezzo}
    \left(\Cl(M_{\vv}(S,H))_{\mathbb{Q}}/\Pic(M_{\vv}(S,H))_{\mathbb{Q}}\right)^\tau=0.
  \end{equation}
  Since the action of \(\tau\) is defined over \(\mathbb{Q}\), it is enough to prove the corresponding statement after tensoring with \(\mathbb{R}\).

  As in the proof of Proposition~\ref{prop:geometry-Mv-delPezzo}, choose a \(\vv\)-generic polarisation \(H^+\) sufficiently close to \(H\) so that
  \[
    H^-\coloneqq 2\frac{H^+\cdot H}{H^2}H-H^+\in \mathrm{NS}(S)_\mathbb{R}
  \]
  is ample, and hence a \(\vv\)-generic polarisation in the reflected adjacent chamber.
  Let
  \[
    \phi_\pm\colon M_{\vv}(S,H^\pm)\to M_{\vv}(S,H)
  \]
  be the corresponding small crepant chamber-to-wall contractions. By Corollary~\ref{cor:duality-reflection-delPezzo}, \(\tau\) induces an isomorphism \(\tau\colon M_{\vv}(S,H^+)\iso M_{\vv}(S,H^-)\), compatible with the involution of \(M_{\vv}(S,H)\). Since \(\phi_\pm\) are small, strict transform gives isomorphisms
  \begin{equation}\label{eq:tau-equivariant-isomorphisms-relative-NS-delPezzo}
    \phi_\pm^\ast \colon
    \left(\Cl(M_{\vv}(S,H))_{\mathbb{Q}}/\Pic(M_{\vv}(S,H))_{\mathbb{Q}}\right)\otimes_{\mathbb{Q}}\mathbb{R}
    \iso N^1(M_{\vv}(S,H^\pm)/M_{\vv}(S,H)),
  \end{equation}
  compatible with the action of \(\tau\).

  If \(T\cong \mathbb{P}^1\times \mathbb{P}^1\), there is only one wall, and the relative N\'eron--Severi spaces are one-dimensional. Since \(\tau\) swaps the relative nef cones of \(M_\vv(S, H^\pm)\), it acts as \(-1\) on
  \[
    \left(\Cl(M_{\vv}(S,H))_{\mathbb{Q}}/\Pic(M_{\vv}(S,H))_{\mathbb{Q}}\right)\otimes_{\mathbb{Q}}\mathbb{R},
  \]
  and we deduce~\eqref{eq:Q-factoriality-criterion-delPezzo} in this case. 
  
  Assume now that \(T\cong \mathrm{Bl}_{r-1}\mathbb{P}^2\), with \(r=10-d\) and \(2\leq d\leq 8\). In these cases \(\vv\in \widetilde{H}_\mathrm{alg}(S, \mathbb{Z})\) is primitive, so the morphisms \(\phi_\pm\) are crepant resolutions, but \(H\) lies on the intersection of several $\mathbf{v}$-walls. Consider the Mukai morphisms
  \[ \lambda_\vv^\pm\colon \vv^\perp \iso H^2(M_{\vv}(S,H^\pm), \mathbb{Z}), \]
  as in \textsection\ref{subsec:2nd-cohomology-moduli-sheaves}. In the proof of Proposition~\ref{prop:geometry-Mv-delPezzo} we showed that \(\lambda_\vv^\pm\) induce natural isomorphisms
  \[
    \mathrm{NS}(S)_\mathbb{R} \supset H^\perp_{\mathbb{R}}\iso N^1(M_{\vv}(S,H^\pm)/M_{\vv}(S,H)), \qquad
    \delta\mapsto \lambda_{\vv}^\pm(0,\delta,0).
  \]
  It remains to compare these identifications under \(\tau\). We claim that, for every \(\delta\in H^\perp_{\mathbb{R}}\),
  \[
    \tau^\ast\lambda_{\vv}^-(0,\delta,0)=-\lambda_{\vv}^+(0,\delta,0).
  \] 
  Equivalently, the following diagram commutes:
  \[
    \begin{tikzcd}
      H^\perp_{\mathbb{R}} \arrow[r, "\sim"] \arrow[d, "-\mathrm{id}"] & N^1(M_{\vv}(S,H^-)/M_{\vv}(S,H)) \arrow[d, "\tau^\ast"] \\
      H^\perp_{\mathbb{R}} \arrow[r, "\sim"] & N^1(M_{\vv}(S,H^+)/M_{\vv}(S,H))
    \end{tikzcd}
   \]
  By Remark~\ref{rmk:tau-chambers-mukai-morphisms}, applied with \(\alpha=H^+\) and \(\beta=H^-\) and extended by linearity to \(\vv^\perp_{\mathbb{R}}\), we have
  \[
    \tau^\ast\lambda_{\vv}^-(\mathbf{w})=\lambda_{\vv}^+\left(\iota^\ast(\mathbf{w}^\vee\cdot e^{-H})\right),
    \qquad \mathbf{w}\in \vv^\perp_{\mathbb{R}}.
  \]
  For \(\mathbf{w}=(0,\delta,0)\), with \(\delta\in H^\perp_{\mathbb{R}}\), we have \(\iota^\ast\delta=\delta\) and \(\delta\cdot H=0\). Hence
  \[
    \iota^\ast(\mathbf{w}^\vee\cdot e^{-H})=(0,-\delta,0),
  \]
  proving the claim. By the normalisation of the Mukai morphisms in \textsection\ref{subsec:2nd-cohomology-moduli-sheaves}, the classes \(\lambda_\vv^\pm(0,\delta,0)\) agree on the common open subset \(M_\vv^s(S,H)\), and hence correspond under strict transform. Together with the $\tau$-equivariance of~\eqref{eq:tau-equivariant-isomorphisms-relative-NS-delPezzo}, this shows that \(\tau\) acts as \(-\mathrm{id}\) on the \(\mathbb{Q}\)-factoriality defect space
  \[
    \left(\Cl(M_{\vv}(S,H))_{\mathbb{Q}}/\Pic(M_{\vv}(S,H))_{\mathbb{Q}}\right)\otimes_{\mathbb{Q}}\mathbb{R},
  \]
  so~\eqref{eq:Q-factoriality-criterion-delPezzo} holds.

  \smallskip
  We turn to Betti and Picard numbers. For \(d\geq 2\), Theorem~\ref{thm:trivial-action-cohomology} and Proposition~\ref{prop:geometry-Mv-delPezzo} give
  \[
    b_2(X)=b_2(M_{\vv}(S,H))=14+d,\qquad \rho(X)=\rho(M_{\vv}(S,H))=2.
  \]
  For \(d=1\), the description of symplectic involutions on manifolds of \(\mathrm{K3}^{[2]}\)-type~\cite[Thm.\@ 4.1]{Mongardi2012SymplecticInvolutions} gives \(\rk H^2(M_\vv(S,H),\mathbb{Q})^\tau=15\), hence \(b_2(X)=15\). For the Picard number, pull-back by \(q\) gives
  \[
    q^\ast\NS(X)_{\mathbb{Q}}=\NS(M_\vv(S,H))_{\mathbb{Q}}^\tau.
  \]
  It remains to compute the invariant algebraic classes upstairs. Since there are no \(\vv\)-walls, we may apply Remark~\ref{rmk:tau-chambers-mukai-morphisms} with \(\alpha=\beta=H\). For \(\mathbf{w}=(r,c,s)\in \vv^\perp\cap\widetilde H_{\mathrm{alg}}(S,\mathbb{Z})\), the action induced through \(\lambda_\vv\) as in~\eqref{eq:action-tau-vperp-Mukai-morphism} fixes \(\mathbf{w}\) if and only if \(r=-c\cdot H\) and \(2c=(c\cdot H)H\). Since \(H\) is primitive in \(\NS(S)\), the fixed integral algebraic sublattice is given by
  \[
    \mathbb{Z}(-2,H,0)\oplus \mathbb{Z}(0,0,1),
  \]
  hence \(\rho(X)=2\).

  \smallskip
  Finally, we describe the singularities. Since \(M_{\vv}(S,H)^{\mathrm{sing}}\) has codimension at least \(6\) for \(d\geq 2\), and is empty for \(d=1\), the codimension-\(2\) singularities of the quotient arise from the codimension-\(2\) fixed locus of \(\tau\). By Proposition~\ref{prop:fixed-locus-tau-delPezzo}, this has a unique irreducible component, namely the relative Prym component \(P_{\vv}(S,H)\), lying over \(f^\ast|-K_T|\). At a general point \(p\in P_{\vv}(S,H)\), the moduli space \(M_{\vv}(S,H)\) is smooth, and the fixed component is smooth of codimension \(2\). Thus, in \'etale local coordinates, the action of \(\tau\) is trivial along \(T_pP_{\vv}(S,H)\) and acts as \(-\mathrm{id}\) on the two-dimensional normal space. The quotient is therefore locally, up to a smooth factor, the surface singularity \(\mathbb{A}^2/\{\pm1\}\), i.e.\@ an \(A_1\)-singularity. Hence \(q(P_{\vv}(S,H))\) is the unique codimension-\(2\) component claimed above. Moreover, the isolated fixed points \(\Pic^0(R)[2]\) over the ramification curve \(R\in |H|\) give isolated quotient singularities over the singular point of \(|H|/\langle\iota\rangle\).
\end{proof}

\subsection{Terminal models of the quotient}

Set \(\Sigma\coloneqq q(P_{\vv}(S,H))\). Along the general point of \(\Sigma\), the quotient has a transverse \(A_1\)-singularity. Since this is the only codimension-\(2\) component of \(X^{\mathrm{sing}}\), any \(\mathbb{Q}\)-factorial terminalisation has a single exceptional divisor over \(\Sigma\); the isolated quotient singularities remain terminal and ensure that the resulting space is still singular. The following result is the precise form of Theorem~\ref{thm:intro-main-examples} from the Introduction.

\begin{theorem}
  \label{thm:terminalisation-of-X}
  There is a unique \(\mathbb{Q}\)-factorial terminalisation \(\phi\colon \widetilde{X}\to X\)
  up to isomorphism over \(X\). Moreover, \(\widetilde{X}\) is an irreducible symplectic variety of dimension \(2d+2\) with
  \[ b_2(\widetilde{X})=15+d, \quad \rho(\widetilde{X})=3. \]
  It is endowed with a Lagrangian fibration to a weighted projective space
  \begin{equation*}
    \pi\circ\phi\colon \widetilde{X} \to \mathbb{P}(1^{d+1}, 2).
  \end{equation*}
  Its singular locus is non-empty and contains at least \(2^{2d+2}\) isolated points lying over the singular point of \(\mathbb{P}(1^{d+1}, 2)\).
\end{theorem}

\begin{proof}
  By Proposition~\ref{prop:geometry-X}, the quotient \(X\) is an irreducible symplectic variety of dimension \(2d+2\), hence has canonical singularities. Let \(\phi\colon \widetilde{X}\to X\) be a \(\mathbb{Q}\)-factorial terminalisation of \(X\)~\cite[Cor.\@ 1.4.3]{BirkarEtAl2010MinimalModelsLogGeneralType}. Since \(\phi\) is crepant, \(\widetilde{X}\) is again an irreducible symplectic variety by Proposition~\ref{prop:criteria-ISV}. The Lagrangian fibration on \(X\) pulls back to
  \[
    \pi\circ\phi\colon \widetilde{X}\to \mathbb{P}(1^{d+1},2).
  \]
  Since \(\phi_\ast\mathscr{O}_{\widetilde{X}}\cong \mathscr{O}_X\), the fibres remain connected. Over a general point of the complement of \(f^\ast|-K_T|\cong \mathbb{P}^d\subset\mathbb{P}(1^{d+1},2)\), the morphism \(\phi\) is an isomorphism on the fibre, so the fibre coincides with the corresponding Lagrangian fibre of \(X\). Hence the pulled-back fibration is again Lagrangian.

  We compute the exceptional contribution. By Proposition~\ref{prop:geometry-X}, the quotient \(X\) has a unique codimension-\(2\) component of the singular locus, namely \(\Sigma\). At a general point of \(\Sigma\), the germ of \(X\) is, up to a smooth factor, the surface singularity \(\mathbb{A}^2/\{\pm1\}\). The unique terminalisation of this transverse \(A_1\)-singularity extracts exactly one divisor. By semismallness of \(\phi\)~\cite[Prop.\@ 2.16]{Tighe2025LooijengaLuntsVerbitskyAlgebra}, every exceptional divisor of \(\phi\) must dominate a codimension-\(2\) component of \(X^{\mathrm{sing}}\). Thus \(\phi\) has precisely one irreducible exceptional divisor, say \(E\). Since \(X\) and \(\widetilde{X}\) are \(\mathbb{Q}\)-factorial, the classes of exceptional prime divisors form a basis of \(N^1(\widetilde{X}/X)\)~\cite[Lem.\@ 3.39]{KollarMori1998BirationalGeometry}. Hence
  \[
    \rho(\widetilde{X}/X)=1.
  \]
  The relative N\'eron--Severi exact sequence, together with \(\rho(X)=2\) by Proposition~\ref{prop:geometry-X}, then gives
  \[
    \rho(\widetilde{X})=\rho(X)+\rho(\widetilde{X}/X)=3.
  \]

  Furthermore, the pull-back \(\phi^\ast\colon H^2(X,\mathbb{Q})\to H^2(\widetilde{X},\mathbb{Q})\) is injective and induces an isomorphism on transcendental lattices~\cite[Lemma 2.1]{BakkerLehn2021GlobalTorelliSingularSymplectic}. Its cokernel is therefore algebraic~\cite[Lemma 5.21]{BakkerLehn2022GlobalModuli} and generated by the class of the unique exceptional divisor \(E\), so
  \[
    b_2(\widetilde{X})=b_2(X)+1=15+d.
  \]

  For uniqueness, any two crepant terminal models are connected over \(X\) by a sequence of flops~\cite[Thm.\@ 1]{Kawamata2008FlopsConnectMinimalModels}. Since \(\rho(\widetilde{X}/X)=1\) and \(\phi\) is divisorial, there is no small contraction over \(X\), and no non-trivial flop can occur. Thus \(\phi\) is unique up to isomorphism over \(X\).

  Finally, consider one of the isolated quotient singularities \(p\in q(\Pic^0(R)[2])\) from Proposition~\ref{prop:geometry-X}. The moduli space \(M_{\vv}(S,H)\) is smooth at the corresponding isolated fixed point, and the linearised action of \(\tau\) on the tangent space is \(-\mathrm{id}\). Thus the germ of \(X\) at \(p\) is
  \[
    (X, p)\cong (\mathbb{A}^{2d+2}/\{\pm1\}, 0).
  \]
  Since \(2d+2\geq 4\), this quotient singularity is terminal. Therefore \(\phi\) is an isomorphism near these points, and they remain isolated singular points of \(\widetilde{X}\). 
\end{proof}

The proof also identifies the rational transcendental Hodge structure of the terminal model. Via pull-back by the finite quotient map \(q\), one has an isomorphism of rational Hodge structures
\begin{equation}\label{eq:cohomology-quotient-invariant-part}
  q^\ast\colon H^2(X,\mathbb{Q})\iso
  H^2(M_\vv(S,H),\mathbb{Q})^\tau.
\end{equation}
If \(d\geq 2\), the target is all of \(H^2(M_\vv(S,H),\mathbb{Q})\) by Theorem~\ref{thm:trivial-action-cohomology}. If \(d=1\), the equality with the full second cohomology of \(M_\vv(S,H)\) no longer holds. However, the involution \(\tau\) is symplectic, so the anti-invariant part has no \((2,0)\)-part and is therefore algebraic. Since moreover \(q^\ast\NS(X)_{\mathbb{Q}}=\NS(M_\vv(S,H))_{\mathbb{Q}}^\tau\), in both cases \eqref{eq:cohomology-quotient-invariant-part} restricts to an isomorphism
\[
  q^\ast \colon T(X)_{\mathbb{Q}}\iso T(M_\vv(S,H))_{\mathbb{Q}}.
\]
Together with the identification~\eqref{eq:transcendental-lattice-Mv-delPezzo} and the isomorphism on transcendental lattices induced by \(\phi^\ast\) in the proof of Theorem~\ref{thm:terminalisation-of-X}, this yields an isomorphism of rational Hodge structures
\[
  T(S)_{\mathbb{Q}} \cong T(\widetilde{X})_{\mathbb{Q}}.
\]

Note that the corresponding integral statement is much subtler, as $q^\ast\colon H^2(X,\mathbb{Z})\to H^2(M_\vv(S,H),\mathbb{Z})^\tau$ is not necessarily an isomorphism on integral cohomology. 

\begin{remark}[Recovering the terminalisation]
The uniqueness statement shows that the terminalisation is canonically determined as a birational model over \(X\). It is natural to ask whether this model can be recovered directly from the codimension-\(2\) stratum \(\Sigma\subset X^{\mathrm{sing}}\).
Let \(D\subset X^{\mathrm{sing}}\) be the \emph{dissident locus}~\cite{Namikawa2001DeformationSymplecticSingularities}, namely the union of the higher-codimensional singular strata. Set \(U\coloneqq X\setminus D\), and let \(j\colon U\hookrightarrow X\) be the open immersion. On \(U\), the terminalisation is the blow-up of the reduced codimension-\(2\) component of \(X^{\mathrm{sing}}\); see for instance~\cite[Prop.\@ 3.7]{BertiniEtAl2025TerminalizationsQuotients}:
\[
  \widetilde{X}\times_X U\cong \operatorname{Bl}_{\Sigma\cap U}U
  =\Proj_U\bigoplus_{k\geq 0}\mathscr{I}_{\Sigma\cap U}^k.
\]
Following the idea outlined in~\cite[Rem.\@ 4.6]{KaledinLehn2007LocalStructureOGrady}, the terminalisation can be recovered by extending a suitable Veronese of this Rees algebra across \(D\). Indeed, let \(E\subset\widetilde{X}\) be the unique exceptional divisor. Choose \(m>0\) sufficiently divisible and sufficiently large, so that \(mE\) is Cartier, \(-mE\) is sufficiently \(\phi\)-ample, and the corresponding Veronese algebra agrees with the Rees algebra over \(U\). Then
\[
  \widetilde{X}\cong \Proj_X\bigoplus_{k\geq 0}\phi_\ast\mathscr{O}_{\widetilde{X}}(-kmE).
\]
Since \(\phi\) is semismall~\cite[Prop.\@ 2.16]{Tighe2025LooijengaLuntsVerbitskyAlgebra} and \(D\) has codimension at least \(4\), the inverse image \(\phi^{-1}(D)\) has codimension at least \(2\) in \(\widetilde{X}\). Hence sections of \(\mathscr{O}_{\widetilde{X}}(-kmE)\) are determined by their restriction over \(U\). By the choice of \(m\), this gives isomorphisms
\[
  j_\ast\mathscr{I}_{\Sigma\cap U}^{km}\iso \phi_\ast\mathscr{O}_{\widetilde{X}}(-kmE), \quad k\geq 0.
\]
Thus
\[
  \widetilde{X}\cong
  \Proj_X\bigoplus_{k\geq 0}j_\ast\mathscr{I}_{\Sigma\cap U}^{km}.
\]
In this sense the terminalisation is recovered intrinsically from \(X\), \(\Sigma\), and the dissident locus \(D\). Since \(X\) is normal and \(\Sigma\) is integral, one has \(j_\ast\mathscr{I}_{\Sigma\cap U}=\mathscr{I}_{\Sigma}\). The remaining question is whether one can take \(m=1\) and whether the saturated algebra is generated in degree one; equivalently, whether \(\widetilde{X}\) is the ordinary blow-up of \(X\) along \(\Sigma\) with its reduced structure. This depends on the local models along the deeper strata, in particular along
\[
  q(M_{\vv}(S,H)^{\mathrm{sing}})\subset D\cap\Sigma,
\]
and should be approachable through the corresponding Ext-quiver descriptions. One may then ask how far the reduced blow-up description persists along the deeper strata, and how far the resulting model remains smooth over \(\Sigma\). Example~\ref{ex:generic-inherited-singularities} verifies this at the general points of the singular locus inherited from \(M_\vv(S,H)\).
\end{remark}

\begin{example}\label{ex:generic-inherited-singularities}
At general points of the singular locus inherited from \(M_\vv(S,H)\), namely along the shallowest strata of the dissident locus, the preceding local analysis can be made explicit. Let \(\mathfrak{t}\coloneqq (1,\vv_1;1,\vv_2)\) be a minimal polystable type as in Proposition~\ref{prop:polystable-strata-delPezzo}, and choose a general point \([\mathscr{F}]\in M_\vv^\mathfrak{t}(S,H)\). By the computation in the proof of Proposition~\ref{prop:fixed-locus-tau-delPezzo}, there is an isomorphism of germs
  \[
    (M_\vv(S,H),[\mathscr{F}])\cong
    (\overline{\mathcal{O}}^{\mathfrak{sl}_4}_{\mathrm{min}}\times \mathbb{C}^{2d-4},0).
  \]
In these coordinates, \(\tau\) fixes the smooth factor and acts on \(\overline{\mathcal{O}}^{\mathfrak{sl}_4}_{\mathrm{min}}\) by
  \[
    A\mapsto -JA^tJ^{-1},
  \]
where \(J\) is the standard symplectic form on \(\mathbb{C}^4\). On the transverse factor, the quotient is induced by \(A\mapsto A-JA^tJ^{-1}\in \mathfrak{sp}_4\); by~\cite{AbuafCarini2024SemistableHiggs}, this gives
\[
  \overline{\mathcal{O}}^{\mathfrak{sl}_4}_{\mathrm{min}}/\langle\tau\rangle
  \cong
  \overline{\mathcal{O}}^{\mathfrak{sp}_4}_2\subset \mathfrak{sp}_4
\]
where the right-hand side is the closure of the rank-\(2\) nilpotent orbit. Hence
  \[
    (X,q([\mathscr{F}]))\cong
    (\overline{\mathcal{O}}^{\mathfrak{sp}_4}_2\times \mathbb{C}^{2d-4},0).
  \]
Since \(\overline{\mathcal{O}}^{\mathfrak{sp}_4}_2\) is crepantly resolved by blowing up its reduced singular locus~\cite{AbuafCarini2024SemistableHiggs}, \(\phi\) is locally the blow-up of \(X\) along \(\Sigma\) near this point, and it is a resolution of singularities there.
\end{example}

\begin{remark}[Deformation types]\label{rmk:novelty-comparison}
The four-dimensional case \(d=1\) belongs to the known deformation class of Nikulin-type orbifolds, which arise as terminal models of quotients of fourfolds of \(\mathrm{K3}^{[2]}\)-type by symplectic involutions~\cite{Menet2015BeauvilleBogomolovLattice, CamereGarbagnatiKapustkaKapustka2024ProjectiveOrbifoldsNikulinType}.

For \(d\geq 2\), Theorem~\ref{thm:terminalisation-of-X} gives singular irreducible symplectic varieties with
\[
  (\dim \widetilde X,b_2(\widetilde X))=(2d+2,15+d),
  \qquad d=2,\ldots,9.
\]
To the best of our knowledge, these invariants do not occur among known examples: neither among Menet's higher-dimensional Fujiki examples~\cite{Menet2022DeformationClassesHKOrbifolds}, nor among the quotients classified in~\cite{BertiniEtAl2025TerminalizationsQuotients}.

The extremal case \(d=9\) gives a \(20\)-dimensional example with \(b_2=24\); this is currently the largest known second Betti number for irreducible symplectic varieties. This value is also achieved by O'Grady's smooth tenfold~\cite{OGrady1999DesingularizedModuliK3} and lies in the same large-\(b_2\) range as the \(42\)-dimensional relative compactified Jacobian of~\cite{LiuLiuXu2025IrreducibleSymplecticVarieties}, which has \(b_2\geq 24\). In forthcoming work~\cite{YamagishiLinInPreparationDualityInvolution}, another \(20\)-dimensional example with \(b_2=24\) and \(2^{20}\) isolated singular points is constructed as a quotient of a moduli space of vector bundles on a genus-\(2\) K3 surface; it is reasonable to expect that the two examples are closely related, likely via an autoequivalence on the K3 side.

The two degree-eight cases both have dimension \(18\) and \(b_2=23\). It would be interesting to determine whether the terminalisations are deformation equivalent. Before terminalisation, the corresponding quotients are not: the singularities inherited from the wall moduli space \(M_\vv(S,H)\), encoded by the polystable-type Hasse diagrams of Figures~\ref{fig:singular-locus-Bl1P2} and~\ref{fig:singular-locus-P1xP1}, already separate the blow-up and quadric cases.
\end{remark}

\pagestyle{plain}
\bibliographystyle{alpha}
\bibliography{references}
\end{document}